\documentclass[12pt, oneside]{amsart}

\usepackage[utf8]{inputenc}
\usepackage{csquotes}

\usepackage{amssymb}
\usepackage{amsthm}
\usepackage{amsmath}
\usepackage{mathrsfs}
\usepackage{graphicx}

\usepackage[a4paper,left=2cm,top=3cm,right=2cm,bottom=2cm]{geometry}

\usepackage{tikz-cd}
\usepackage{xcolor}

\usepackage{enumerate}

\usepackage{indentfirst}

\usepackage{hyperref}
\hypersetup{colorlinks=true,allcolors=black}

\usepackage{graphicx}

\DeclareMathOperator{\differential}{d}

\DeclareMathOperator{\codim}{codim}

\DeclareMathOperator{\id}{Id}

\DeclareMathOperator{\tub}{Tub}

\DeclareMathOperator{\im}{im}

\DeclareMathOperator{\reg}{reg}

\DeclareMathOperator{\Ho}{H}

\newcommand{\blup}[2]{\mathrm{Bl}_{#1}{#2}}

\mathchardef\ordinarycolon\mathcode`\:
\mathcode`\:=\string"8000
\begingroup \catcode`\:=\active
  \gdef:{\mathrel{\mathop\ordinarycolon}}
\endgroup

\begin{document}

    \title{Blown-up singular Riemannian foliations}

\author{Francisco C.~Caramello Jr.}
\address{Departamento de Matemática, Universidade Federal de Santa Catarina, R. Eng. Agr. Andrei Cristian Ferreira, 88040-900, Florianópolis - SC, Brazil}
\email{francisco.caramello@ufsc.br}
\author{Laura Ribeiro dos Santos}
\address{Instituto de Matemática e Estatística, Universidade de São Paulo, Rua do Matão, 1010, 05508-090, São Paulo - SP, Brazil}
\address{Departamento de Matemática, Universidade Federal de Santa Catarina, R. Eng. Agr. Andrei Cristian Ferreira, 88040-900, Florianópolis - SC, Brazil}
\email{laura.ribeiro.santos@usp.br}

\newenvironment{proofoutline}{\proof[Sketch]}{\endproof}
\theoremstyle{definition}
\newtheorem{example}{Example}[section]
\newtheorem{definition}[example]{Definition}
\newtheorem{obs}[example]{Remark}
\newtheorem{exercise}[example]{Exercise}
\theoremstyle{plain}
\newtheorem{prop}[example]{Proposition}
\newtheorem{thm}[example]{Theorem}
\newtheorem{lemma}[example]{Lemma}
\newtheorem{cor}[example]{Corollary}
\newtheorem{claim}[example]{Claim}
\newtheorem{thmx}{Theorem}
\renewcommand{\thethmx}{\Alph{thmx}}
\newtheorem{corx}[thmx]{Corollary}

\newcommand{\dif}[0]{\differential\hspace{-1pt}}
\newcommand{\od}[2]{\frac{\dif #1}{\dif #2}}
\newcommand{\pd}[2]{\frac{\partial #1}{\partial #2}}
\newcommand{\dcov}[2]{\frac{\nabla #1}{\dif #2}}
\newcommand{\proin}[2]{\left\langle #1, #2 \right\rangle}
\newcommand{\f}[0]{\mathcal{F}}
\newcommand{\g}[0]{\mathcal{G}}

\newcommand{\metric}{\ensuremath{\mathrm{g}}}
\newcommand{\qcd}{\begin{flushright} $\Box$ \end{flushright}}
\newcommand{\rar}[0]{\rightarrow}
\newcommand{\fol}[0]{\mathcal{F}}
\newcommand{\vct}[0]{\mathfrak{X}}
\newcommand{\trs}[0]{\mathfrak{l}}
\newcommand{\bsc}[0]{\Omega_{\text{bas}}^0(M,\fol)}
\newcommand{\RR}[0]{\mathbb{R}}
\newcommand{\folb}[0]{\mathcal{F}_{\text{bas}}}
\newcommand{\Sg}[0]{\Sigma}


\begin{abstract}
We study the effect of projective blow-ups on the leaf closures
and basic cohomology of singular Riemannian foliations. For a foliation
with smooth Molino sheaf on a compact manifold, we show that the Molino
sheaf of the blow-up is the inverse image of the original sheaf under
blow-down. In particular, desingularization preserves the Killing
property and the structural Lie algebra. We obtain a closedness criterion
in terms of the Euler characteristics of the ambient manifold and the
blow-up centers, assuming finite fundamental group and singular strata
of codimension at least three. For singular Killing foliations on compact
manifolds, we approximate the space of leaf closures in Gromov--Hausdorff
distance by Riemannian orbifolds whose dimension equals the codimension
of the foliation. Finally, relative basic cohomology gives a long exact
sequence relating the original foliation, its blow-up, and their
restrictions to the center and the exceptional divisor.
\end{abstract}
\maketitle
\setcounter{tocdepth}{1}
\tableofcontents

\section{Introduction}

The blow-up construction is a central geometric technique when it comes to resolving singularities. Historically, blow-ups were first introduced in algebraic geometry: given a (complex) subvariety \( S \subset M \), the blow-up \(\operatorname{Bl}_S M\) replaces \(S\) by the projectivization of its normal bundle, producing a new variety in which the singular locus is \enquote{spread out} along an exceptional divisor. This refinement reached a turning point with Hironaka’s celebrated theorem on the resolution of singularities for complex analytic and algebraic varieties (see, for instance, \cite{hironaka}, \cite[Pg.~621]{poag}). In the smooth (real) setting, the blow-up is analogous: given a smooth manifold \(M\) and a closed embedded submanifold \(S \subset M\), the real blow-up replaces \(S\) by its projective normal bundle.

More recently, this technique has been applied in the context of singular Riemannian foliations in \cite{desmolino} and \cite{desmarcos}, following similar results coming from the theory of Lie group actions \cite[Section 2.9]{duistermaat}. These foliations have locally equidistant leaves, a constraint that leads to a robust structural theory, nowadays known as Molino theory in honor to its main contributor P.~Molino \cite{molino}. Molino's conjecture that the leaf closures of such a foliation $\fol$ form another singular Riemannian foliation $\overline{\fol}$, however, was only recently established to hold true, in \cite{molinoconjec}. Part of the structural theory is devoted to the description of $\overline{\fol}$ by a locally constant sheaf $\mathscr{C}_{\fol}$ of Lie algebras of transverse vector fields (see Section \ref{section background} for more details). When $\mathscr{C}_{\fol}$ is smooth and globally constant, ${\mathcal{F}}$ is called a singular \emph{Killing} foliation, following the terminology established in \cite{mozgawa} (and in \cite{equicaramello} for the singular case).

The first application of the blow-up method in this context was by Molino himself, desingularizing a regular Riemannian foliation $\fol$ along the locus $\Sigma$ consisting of the minimal (deepest, most singular) strata of $\overline{\f}$ and obtaining a blown-up foliation $\blup{\Sg}{\fol}$ on $\blup{\Sigma}{M}$ with simpler dynamics. The particular geometry of Riemannian foliations ensures that the lift of $\fol|_{M\setminus\Sigma}$, which is a copy of it, indeed extends well to the exceptional divisor $E=\pi^{-1}(\Sigma)$. In \cite{desmarcos} this construction was generalized to an already singular Riemannian foliation, resulting in a regular foliation $\blup{}{\fol}$ after a finite sequence of blow-ups. Both works focused on the construction of a metric on $\blup{\Sg}{M}$ that is adapted to $\blup{\Sg}{\fol}$, hence turning it into a singular Riemannian foliation. Figure \ref{blup circle} illustrates the blow-up in the simple, yet quintessential case of the foliation of $\mathbb{R}^2$ by the orbits of $\mathrm{SO}(2)$, for which $\blup{}{M}$ is a Möbius band.

\begin{figure}
\centering{
\resizebox{0.3\textwidth}{!}{
%% Creator: Inkscape 1.4.3 (0d15f75, 2025-12-25), www.inkscape.org
%% Accompanies image file 'blupS1.pdf' (pdf, eps, ps)
\begingroup%
  \makeatletter%
  \providecommand\color[2][]{%
    \errmessage{(Inkscape) Color is used for the text in Inkscape, but the package 'color.sty' is not loaded}%
    \renewcommand\color[2][]{}%
  }%
  \providecommand\transparent[1]{%
    \errmessage{(Inkscape) Transparency is used (non-zero) for the text in Inkscape, but the package 'transparent.sty' is not loaded}%
    \renewcommand\transparent[1]{}%
  }%
  \providecommand\rotatebox[2]{#2}%
  \newcommand*\fsize{\dimexpr\f@size pt\relax}%
  \newcommand*\lineheight[1]{\fontsize{\fsize}{#1\fsize}\selectfont}%
  \ifx\svgwidth\undefined%
    \setlength{\unitlength}{134.55458874bp}%
    \ifx\svgscale\undefined%
      \relax%
    \else%
      \setlength{\unitlength}{\unitlength * \real{\svgscale}}%
    \fi%
  \else%
    \setlength{\unitlength}{\svgwidth}%
  \fi%
  \global\let\svgwidth\undefined%
  \global\let\svgscale\undefined%
  \makeatother%
  \begin{picture}(1,2.05230669)%
    \lineheight{1}%
    \setlength\tabcolsep{0pt}%
    \put(0.77452152,0.99858004){\color[rgb]{0,0,0}\makebox(0,0)[lt]{\lineheight{1.25}\smash{\begin{tabular}[t]{l}$\blup{}{M}$\end{tabular}}}}%
    \put(0,0){\includegraphics[width=\unitlength,page=1]{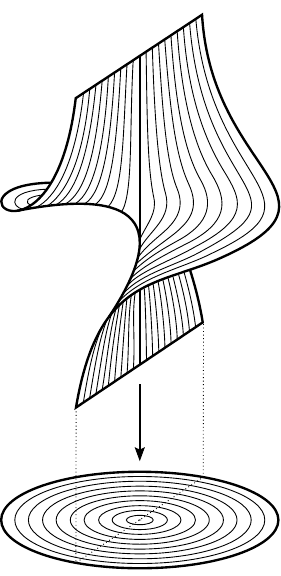}}%
    \put(0.52135159,0.55339694){\color[rgb]{0,0,0}\makebox(0,0)[lt]{\lineheight{1.25}\smash{\begin{tabular}[t]{l}$\pi$\end{tabular}}}}%
    \put(0.3806965,1.97573021){\color[rgb]{0,0,0}\makebox(0,0)[lt]{\lineheight{1.25}\smash{\begin{tabular}[t]{l}$E$\end{tabular}}}}%
    \put(0.02412133,0.41345242){\color[rgb]{0,0,0}\makebox(0,0)[lt]{\lineheight{1.25}\smash{\begin{tabular}[t]{l}$\Sg$\end{tabular}}}}%
    \put(0.8512834,0.00871021){\color[rgb]{0,0,0}\makebox(0,0)[lt]{\lineheight{1.25}\smash{\begin{tabular}[t]{l}$M$\end{tabular}}}}%
    \put(0,0){\includegraphics[width=\unitlength,page=2]{blupS1.pdf}}%
  \end{picture}%
\endgroup%
}}
\caption{Blow-up of the foliation of $\RR^2$ by orbits of $\mathrm{SO}(2)$.}
\label{blup circle}
\end{figure}

In the present paper, we find new applications of the blow-up method and establish further properties of it, aiming to generalize results from the regular setting. For instance, we provide a natural alternative construction of the foliation $\blup{}{\fol}$ by extending the leaf-tangent distribution across the exceptional divisor, and identify this construction
with Alexandrino's. We then generalize results of \cite{desmolino}, showing that the dynamics of $\blup{}{\fol}$ are described by the blow-up of the Molino (central) sheaf $\mathscr{C}_{\fol}$ of $\fol$. More precisely, lifting transverse fields gives the following comparison of Molino sheaves:

\begin{thmx}[Theorem~\ref{blup sheaf}]\label{thmA}
Let $\fol$ be a singular Riemannian foliation, with smooth Molino sheaf, of a compact manifold $M$, and let $\pi\colon\blup{\Sg}{M}\to M$ be the blow-down along its minimal locus. The lift of transverse fields induces an isomorphism
\[\mathscr C_{\blup{\Sg}{\fol}}\cong\pi^{-1}\mathscr C_\fol.\]
In particular, $\mathscr C_{\blup{\Sg}{\fol}}$ is smooth.
\end{thmx}

For a singular Killing foliation, iteration then gives a regular Killing foliation with the same structural Lie algebra $\mathfrak{g}_\fol$ (Corollary~\ref{corollary: desing killing is killing}). For a smooth Molino sheaf with finite monodromy, a finite cover of the desingularization carries a Killing foliation. This gives the following closedness criterion.

\begin{thmx}[Theorem~\ref{closed leaves kill}]\label{thmB}
Let $\fol$ be a singular Riemannian foliation with smooth Molino sheaf on a compact, connected manifold $M$ with finite fundamental group. Suppose that every singular stratum has codimension at least three. Let $\Sg_1,\dots,\Sg_\ell$ be the even-codimensional minimal strata appearing in the blow-up process. If
\[\chi(M)\neq\sum_{j=1}^{\ell}\chi(\Sg_j),\]
then every leaf of $\fol$ is closed.
\end{thmx}

In particular, if all singular strata have odd codimension, the sum in Theorem~\ref{thmB} is empty and a nonzero
$\chi(M)$ suffices under the remaining hypotheses (Corollary~\ref{cor odd dimensions}). Corollary~\ref{local closed leaf criterion} gives a separate criterion using a single closed regular leaf.

For closed singular Riemannian foliations, Alexandrino obtained Gromov--Hausdorff approximations of the leaf space by orbifolds \cite[Theorem~1.5 and Corollary~1.6]{desmarcos}. We control the distortion of the blow-down on ambient distances, which also gives a comparison of the spaces of leaf closures (Proposition \ref{blow-down metric approximation}). For singular Killing foliations, we combine this estimate with closed deformations of the regular desingularization to obtain decollapsed approximants \cite[Theorem~5.2]{chicos}.

\begin{thmx}[Theorem~\ref{chicos def}]\label{thmD}
Let $\fol$ be a singular Killing foliation of codimension $d$ on a compact, connected manifold $M$. There are closed Riemannian foliations $\mathcal G_i$ of codimension $d$ on the desingularization $\blup{}{M}$, with bundle-like metrics for which
\[\frac{\blup{}{M}}{\mathcal G_i} \xrightarrow{\mathrm{GH}} \frac{M}{\overline{\fol}}.\]
In particular, the approximating quotients are $d$-dimensional Riemannian orbifolds.
\end{thmx}

Finally, relative basic cohomology relates the original foliation and its blow-up
to their restrictions to the center and the exceptional divisor (Proposition~\ref{relative blowup iso}), which leads to the following:

\begin{thmx}[Theorem~\ref{principal result}]\label{principal result intro}
Let $(M,\fol)$ be a complete singular Riemannian foliation with minimal locus $\Sg$ and let $\pi\colon\blup{\Sg}{M}\to M$ be the blow-down map. Write $i_\Sg$ and $i_E$ for the inclusions of the center and the exceptional divisor. Then there is a natural long exact sequence
\[\cdots\longrightarrow\Ho^i(\fol) \xrightarrow{(\pi^*,\,i_\Sg^*)} \Ho^i(\blup{\Sg}{\fol})\oplus\Ho^i(\fol|_\Sg) \xrightarrow{i_E^*-(\pi|_E)^*}\Ho^i(\blup{\Sg}{\fol}|_E)\longrightarrow\cdots.\]
Moreover, if $\pi^*$ is injective in degrees $i$ and $i+1$, then
\begin{equation}\label{equation isom cohom bas intro}
\Ho^i(\blup{\Sg}{\fol})\cong\pi^*(\Ho^i(\fol))\oplus\frac{\Ho^i(\blup{\Sg}{\fol}|_E)}{(\pi|_E)^*(\Ho^i(\fol|_\Sg))}
\end{equation}
as vector spaces.
\end{thmx}

The decomposition is the basic cohomology analogue of the classical blow-up formula \cite{gitler,poag}. The exact sequence also gives a formula for the basic Euler characteristic (Corollary~\ref{basic euler blowup}). The pullback is always an isomorphism in degree one and injective in degree two, and it is an isomorphism of graded algebras when every component of the minimal locus has codimension two or three (Corollaries~\ref{blowup low basic degrees} and~\ref{blowup low codimension}).

\section{Background} \label{section background}

In this section we briefly review the foundational theory for singular Riemannian foliations. We begin by recalling the notions of basic cohomology, the homothetic transformation lemma and the stratification associated with the foliation. We then introduce Killing foliations and provide an overview of Molino’s theory for both regular and singular cases.

\subsection{Singular Riemannian Foliations} Regular Riemannian foliations are relatively well-understood geometric structures that received extensive attention in the 1980s. The contributions of P. Molino, in particular, provided a deep understanding of the dynamical behavior of the leaves of a regular Riemannian foliation, culminating in what is now known as Molino theory (see \cite[Chapter 5]{molino} or \cite[Chapter 4]{mrcun} for thorough introductions). In the past two decades, the study of singular Riemannian foliations (a generalization of Molino’s framework to the singular setting) has developed substantially, and recent research has yielded new insights into fundamental questions, as well as significant progress on long-standing conjectures and extensions of classical results. In this section we present the basic and fundamental theory of singular Riemannian foliations, presenting the results of regular Riemannian foliations as a particular case.

\begin{definition}[Singular Riemannian Foliations]\label{srf definition} Given $(M,\metric)$ a Riemannian manifold and $\fol$ a partition of $M$ into immersed, connected submanifolds, called \emph{leaves}, we say that $\fol$ is a \emph{singular Riemannian foliation} if
\begin{enumerate}
    \item it is a \emph{singular foliation}: for every $v \in T_x\fol$, there exists $X \in \mathfrak{X}(M)$ always tangent to the leaves and such that $X(x)=v$ (or, in other words, the module $\mathfrak{X}(\fol)$ of vector fields tangent to $\f$ is \textit{transitive});
    \item it is a \emph{transnormal system}, in the terminology of Bolton \cite{bolton}: every geodesic that is perpendicular to a leaf in some point remains perpendicular to all the leaves it meets (or, more succinctly, $\metric$ is \textit{adapted} to $\fol$).
\end{enumerate}
Given $x\in M$, we will denote the leaf containing it by $L_x$.
\end{definition}

We define the \emph{dimension of $\fol$} as 
\[\dim(\fol) = \max_{L\in \fol} \dim(L),\]
and its \emph{codimension} by $\codim(\fol)=\dim(M)-\dim(\fol)$. Let $r$ be the number of dimensions occurring among the leaves of $\fol$. The \emph{depth} of $\fol$ is $\operatorname{depth}(\fol):=r-1$. We say that $\fol$ is a \emph{regular} Riemannian foliation when $\dim(L)$ is constant for every $L \in \fol$, that is, when $\operatorname{depth}(\fol)=0$. In this case one also says that an adapted metric $\metric$ is \emph{bundle-like} for $\fol$. We say that $\fol$ is a \emph{complete} singular Riemannian foliation when there exists a complete metric on $M$ that is adapted to $\f$.

\begin{example}[Homogeneous Singular Riemannian Foliations]\label{homogeneous foliations} Let $H$ be a Lie group acting isometrically on a Riemannian manifold $(M,\metric)$. The induced infinitesimal action $\mu$ of the Lie algebra $\mathfrak{h}$ of the group $H$ generates a distribution $TH$ on $M$, which is precisely the tangent distribution to the orbits of the action: $T_p H(p) = \mu(\mathfrak{h})|_{p}$. This shows that the partition $\fol_{H}$ of $M$ into the connected components of the orbits of $H$ is a singular foliation. Moreover, it is a singular Riemannian foliation, because the action being isometric guarantees that $\metric$ is adapted to $\fol_{H}$ (see, for instance, \cite[Proposition 3.78 and Remark 3.79]{alex}). 
\end{example}

Suppose $(N,\mathcal{G})$ is a singular foliation. A smooth map $f \colon (M,\fol) \rar (N,\mathcal{G})$ is called \emph{foliate} if it maps the leaves of $\fol$ into leaves of $\mathcal{G}$. We denote by $C^\infty(\fol)=\Omega^0(\fol)$ the ring of foliate maps $f\colon (M,\fol)\to \mathbb{R}$ (in which $\mathbb{R}$ is endowed with the trivial foliation by points), whose elements are called \emph{basic functions}. A smooth map $f \colon M \times [0,1] \rar N$ is a \emph{foliate homotopy} when, for every $t \in [0,1]$, the partial maps 
\[\begin{aligned}
    f_t \colon M & \longrightarrow M\\
    x & \longmapsto f(x,t)
\end{aligned}\]
are foliate. If $V \subset M$ is a saturated submanifold (that is, $V$ can be expressed as a union of leaves), then a \emph{foliate deformation retraction $f\colon M \times [0,1] \rar M$} onto $V$ is a foliate homotopy which is also a deformation retraction of $M$ onto $V$. A vector field $X \in \mathfrak{X}(M)$ is called \emph{foliate} when $[X,Y] \in \mathfrak{X}(\fol)$ for every $Y \in \mathfrak{X}(\fol)$. These are precisely the vector fields whose local flows are foliate maps. We denote the Lie algebra of foliate fields by $\mathfrak{L}(\fol)$. The \emph{transverse vector fields} are the elements in the quotient Lie algebra $\mathfrak{l}(\fol):=\mathfrak{L}(\fol)/\mathfrak{X}(\fol)$. These hence combine into the exact sequence of $\Omega^0(\fol)$-modules
\[\begin{tikzcd}
        0 \rar & \mathfrak{X}(\fol) \rar & \mathfrak{L}(\fol) \rar & \mathfrak{l}(\fol) \rar & 0.
\end{tikzcd}\]

A differential form $\omega \in \Omega^k(M)$ is called \emph{basic} if, for every $X \in \mathfrak{X}(\fol)$, it satisfies $\iota_{X}\omega=0$ and $\mathcal{L}_{X}\omega=0$. The space of basic $k$-forms is denoted by $\Omega^k(\fol)$. We denote by
\[\Omega(\fol) := \bigoplus_k \Omega^k(\fol)\]
the \emph{algebra of basic forms of $\fol$}, which is a $\mathbb{Z}$-graded differential algebra (defining $\Omega^k(\fol)=\{0\}$ for $k<0$), since the exterior differential $d$ preserves basic forms (as can be readily seen from Cartan's formula). Therefore we can compute the cohomology groups of the subcomplex of basic forms, which we call the \emph{basic cohomology groups of $\fol$}, denoted by $\Ho^k(\fol)$. We also denote $\Ho(\fol) = \bigoplus_k \Ho^k(\fol)$, which is a graded algebra with the usual exterior product. Notice that, when $\fol$ is the trivial foliation by points, $\Ho(\fol)$ reduces to the de Rham cohomology of $M$ (in general, one intuitively thinks of $\Ho(\fol)$ as the de Rham cohomology of $M / \fol$). When $\dim(\Ho(\fol))<\infty$, the \emph{basic Euler characteristic} of $\fol$ is $\chi_{\mathrm{bas}}(\fol)=\sum_k (-1)^k\dim(\Ho^k(\fol))$. If $f\colon (M, \fol) \to (N, \mathcal{G})$ is foliated, then the pullback of a $\mathcal{G}$-basic form is a $\fol$-basic form, and thus there is an induced linear map 
\[f^*\colon \Ho(\mathcal{G}) \longrightarrow \Ho(\fol).\]
For $\fol$-saturated open sets $U, V \subset M$, the short exact sequence
\[0 \longrightarrow \Omega(\fol|_{U \cup V}) \xrightarrow{(i_U^*,\, i_V^*)} \Omega(\fol|_U) \oplus \Omega(\fol|_V) 
\xrightarrow{i_U^* - i_V^*} \Omega(\fol|_{U \cap V}) \longrightarrow 0\]
induces a Mayer–Vietoris sequence in basic cohomology (see \cite[Section 2.1]{equicaramello}).

\subsection{Homothetic Lemma and Stratification}

Let $(M,\fol,\metric)$ be a complete singular Riemannian foliation and let $P\subset L\in \fol$ be a connected, open subset whose inclusion in $M$ is an embedding. An open neighborhood $U\supset P$ is a \textit{distinguished tubular neighborhood} when it is the diffeomorphic image of $\{v\in\nu P\ {\vert}\ \|v\|<r\}$, for some $r>0$, under the normal exponential map $\exp^\perp:\nu L\to M$, in which case we denote $U=\mathrm{Tub}_r(P)$. We denote by $\rho_P:\mathrm{Tub}_r(P)\to P$ the orthogonal projection. A distinguished tubular neighborhood exists, for instance, when $P$ is relatively compact in $L$.

For all $\lambda\in(0,\infty)$ such that $\mathrm{Tub}_{\lambda r}(P)$ is a distinguished tubular neighborhood of $P$, one can consider the \textit{homothetic transformation}
\[h_\lambda:\mathrm{Tub}_{r}(P)\ni\exp^\perp(v)\longmapsto \exp^\perp(\lambda v)\in \mathrm{Tub}_{\lambda r}(P)\]
around $P$. Each $h_\lambda$ is a diffeomorphism, and $h_\lambda\circ h_\mu=h_{\lambda\mu}$, when both sides of the equation make sense. We can extend the definition to $\lambda=0$, getting $h_0=\rho_P$. Our interest in these maps stem from the following fundamental property of singular Riemannian foliations:

\begin{lemma}[Molino's homothetic lemma {\cite[Lemma 6.2]{molino}}] \label{homothetic lemma} Let $\mathrm{Tub}_r(P)$ be a distinguished tubular neighborhood. Then each $h_\lambda$ is foliate with respect to the restrictions of $\fol$ to $\mathrm{Tub}_{r}(P)$ and $\mathrm{Tub}_{\lambda r}(P)$.
\end{lemma}

We say that the subset $\Sigma^r$ consisting of all leaves of $\fol$ of dimension $r$ is the \emph{locus} of $r$-dimensional leaves. A fundamental consequence of Lemma \ref{homothetic lemma} is that every connected component of it is an embedded submanifold of $M$ \cite[Proposition 6.3]{molino}. Running over all loci, these components are the \emph{strata} of a stratification
\[M = \bigsqcup_{\alpha} \Sigma_\alpha,\]
of $M$. The indexing will be used consistently throughout the paper: superscript for loci, subscript for strata. By construction, for each $\alpha$, the restriction $\fol_\alpha := \fol|_{\Sigma_\alpha}$ forms a regular foliation. The locus of leaves of maximal dimension, $\Sigma^{\dim\fol}$ is the \emph{regular locus} of $\fol$, that we will denote by $\Sigma_{\reg}$. Its complement is the \emph{singular locus} $\Sigma_{\mathrm{sing}}$.

\begin{prop}[{\cite[Section 6.2]{molino}}]\label{proposition: properties strata} With the notation above, the following holds for each $\alpha$:
\begin{enumerate}
    \item A geodesic initially tangent to $\Sigma_\alpha$ and perpendicular to the leaf through its initial point remains in $\Sigma_\alpha$ for sufficiently small times.
    \item $\metric|_{\Sigma_{\alpha}}$ turns $\fol_{\alpha}$ into a regular Riemannian foliation.
    \item If $L \subset \Sigma_{\alpha}$, then $\overline{L} \subset \Sigma_{\alpha}$.
    \item All the singular strata have codimension at least 2, so $\Sigma_{\reg}$ is an open, dense submanifold of $M$.
    \item Lemma \ref{homothetic lemma} extends, \textit{mutatis mutandis}, to the situation in which $L$ is replaced by a saturated closed submanifold $S \subset M$ within a stratum.
\end{enumerate}
\end{prop}

We refer to the singular locus of leaves with minimal dimension as the \emph{minimal locus}. It is usually denoted by $\Sigma_{\min}$, but we will favor the cleaner notation $\Sigma$ whenever this doesn't lead to confusion, since this object will be recurrently used throughout the paper. The lower semicontinuity of $x\mapsto\dim L_x$ makes the minimal locus closed, so it is compact when $M$ is compact.

\subsection{Molino's Theory for Riemannian foliations} Consider a regularly foliated manifold $(M,\fol)$. A bundle-like metric $\metric$ for $\fol$ naturally induces a \emph{transverse metric}, that is, a symmetric, basic $(2,0)$-tensor field $\metric^\intercal$ on $M$ satisfying $\metric^\intercal(v,v)>0$ for every $v$ not tangent to $\fol$ \cite[Proposition 3.5]{molino}. The converse is also true: any transverse metric can be completed to a bundle-like metric \cite[Proposition 3.3]{molino}. Notice that, therefore, the Lie derivative of $\metric^\intercal$ (or, more generally, of any basic tensor field) in the direction of $\overline{X}\in\mathfrak{l}(\fol)$ is well-defined by $\mathcal{L}_{\overline{X}}\metric^\intercal:=\mathcal{L}_{X}\metric^\intercal$.

\begin{definition}[Transverse Killing Field] We say $\overline{X} \in \mathfrak{l}(\fol)$ is a \emph{transverse Killing vector field} when $\mathcal{L}_{\overline{X}}\metric^\intercal=0$.
\end{definition}

Molino's second structural theorem for complete (regular) Riemannian foliations \cite[Theorem 5.2]{molino} describes the behavior of the leaf closures by exhibiting them as orbits of a locally constant sheaf $\mathscr{C}_{\fol}$ of germs of local transverse Killing fields, called the \emph{Molino sheaf} of $\fol$. We refer to \cite{alex4} for a more detailed survey on this topic, including a quick review of sheaf theory. Unwrapping the definitions, this means that, for each $x \in M$, there exists a neighborhood $U$ of $x$ such that
\[T_y \overline{L}_y = T_yL_y \oplus\{X(y) \, | \, X \in \mathscr{C}_{\fol}(U) \},\]
for each $y \in U$. For $M$ connected, the typical stalk of $\mathscr{C}_{\fol}$ is the \textit{structural algebra} $\mathfrak{g}_\fol$ of $\fol$ (here we are taking a shortcut: Molino's original definition of $\mathfrak{g}_\fol$ takes another route and ends up being isomorphic to the opposite Lie algebra of ours).

\begin{definition}[Killing Foliation] A complete (regular) Riemannian foliation is a \emph{Killing foliation} when its Molino sheaf is globally constant.
\end{definition}

This automatically holds, for instance, when $M$ is simply connected. Another important class of examples is the following:

\begin{example}[{\cite[Lemma III]{desmolino}}] \label{homogenous fol is killing}If $\fol$ is the regular Riemannian foliation of a compact manifold $M$ given by the orbits of an isometric action of a Lie group $H$ on $M$, then $\fol$ is a Killing foliation, because the Molino sheaf $\mathscr{C}_{\fol}$ consists of the transverse Killing vector fields induced by the action of $\overline{H}$.   
\end{example}

Let us now cover the singular setting, so assume $\fol$ is a complete \emph{singular} Riemannian foliation. Using its stratification, one can generalize the concept of transverse Killing vector field:

\begin{definition} We say that $X \in \mathfrak{l}(\fol)$ is a \emph{transverse Killing vector field} if its restriction to each stratum $\Sigma_{\alpha}$ is a transverse Killing vector field of the regular Riemannian foliation $\fol_{\alpha}$.
\end{definition}

Extending the Molino sheaf to the singular case is possible, but not straightforward. First, even though a restricted foliation $\fol_{\alpha}$ may fail to be complete, Molino’s theory can still be applied to it. This is possible thanks to the pseudogroup formulation of Molino's theory due to Haefliger and Salem \cite[Appendix D]{molino}, since the holonomy pseudogroup of each $\fol_{\alpha}$ is complete. In particular, for the regular part $\fol_{\reg}$, one can show that the corresponding Molino sheaf $\mathscr{C}_{\reg}$ admits a continuous extension to a locally constant sheaf $\mathscr{C}_{\fol}$ on $M$, referred to as the \emph{Molino sheaf of} $\fol$ \cite[Lemma 6.5]{molino}. Here lies the subtlety: although every section of $\mathscr{C}_{\reg}$ extends continuously through the singular stratum, it is not known whether such an extension is necessarily smooth. Equivalently, a section $\overline{X}$ of $\mathscr{C}_{\fol}$ restricts to a genuine local transverse Killing vector field on each stratum, but is represented on $M$ by a merely continuous local vector field $X$.

Molino proved that the closure of a complete regular Riemannian foliation is a singular Riemannian foliation \cite{molino}, and he conjectured that the closure of a complete singular Riemannian foliation is a singular Riemannian foliation. This conjecture was finally proved to hold true recently in \cite{molinoconjec} by Alexandrino and Radeschi. Nevertheless, the \emph{strong Molino conjecture} is still open, and states that the extension $\mathscr{C}_{\fol}$ described above is always smooth. It is true, for instance, in the particular case of the so-called \emph{orbit-like} foliations, as shown in \cite[Theorem 3.4]{equicaramello}.

\begin{definition}[Killing Singular Foliation] A complete singular Riemannian foliation $\fol$ is a \textit{singular Killing foliation} when the sheaf $\mathscr{C}_{\fol}$ is a globally constant sheaf of Lie algebras of germs of (smooth) transverse Killing vector fields.    
\end{definition}

Notice that we are, hence, including the smoothness of $\mathscr{C}_{\fol}$ as part of the definition of Killing foliations. Again, singular Riemannian foliations of simply connected manifolds are automatically Killing, whenever their sheaves are smooth (e.g., when they are orbit-like). Homogeneous foliations provide another class of examples.

\begin{example} The exact same setting as in Example \ref{homogenous fol is killing} applies without the regularity assumption. Moreover, $\overline{\fol}$ is also homogeneous, given by the connected components of the orbits of $\overline{H}$.    
\end{example}

\begin{obs}\label{remark: sheaf of strata}
For each stratum $\Sigma_{\alpha}$, the corresponding Molino sheaf $\mathcal{C}_{\alpha}$ associated to $\fol_{\alpha}$ is obtained as the quotient of the restriction of $\mathscr{C}_{\fol}$ to $\Sigma_{\alpha}$ by the kernel of the restriction map on sections. Equivalently, this kernel is the subsheaf consisting of sections whose restrictions to $\Sigma_{\alpha}$ vanish \cite[Proposition 6.8]{molino}.
\end{obs}

Consequently, the orbits of $\mathscr{C}_{\fol}$ describe precisely the closures of the leaves of $\fol$, as in the regular case. The typical stalk $\mathfrak{g}_{\fol}$ of $\mathscr{C}_{\fol}$  is isomorphic to the opposite of the structural algebra of $\fol_{\mathrm{reg}}$, which by definition is the structural algebra $\mathfrak{g}_\fol$ of $\fol$. Also from Remark \ref{remark: sheaf of strata}, we see that the structural algebra $\mathfrak{g}_{\alpha}$ of each $\fol_{\alpha}$ is a quotient of $\mathfrak{g}_{\fol}$ \cite[Proposition 6.8]{molino}.

\section{Blow-up of singular Riemannian foliations}\label{alexandrino blow}

We start with a brief recollection of the blow-up construction. For a compact, connected submanifold $S\subset (M,\metric)$, the (projective) blow-up $\blup{S}{M}$ of $M$ along $S$ is usually constructed by first taking a tubular neighborhood $U$ of $S$, considering
\[\hat{U}=\{(x,[\xi])\in U \times \mathbb{P}(\nu S) \mid x=\exp^{\perp}(v)\text{ for some }v\in[\xi],\ \|v\|<\varepsilon\},\]
which comes with the natural projection
\[\begin{aligned}
    \pi \colon \hat{U} & \longrightarrow U,\\
    (x,[\xi]) & \longmapsto x,
\end{aligned}\]
and then gluing $\hat{U}$ to $M\setminus S$ along the common open set $U\setminus S$. The projection $\pi$ extends naturally to the \textit{blow-down map}  $\pi\colon \blup{S}{M}\to M$. The preimage $E:=\pi^{-1}(S)=\mathbb{P}(\nu(S))$ is the \emph{exceptional divisor} of the blow-up, and $\pi \colon E \rar S$ is the canonical projection of the fiber bundle $\mathbb{P}(\nu(S)) \rar S$. This construction can be generalized to a non-compact $S$, provided it admits a suitable neighborhood $U$. We now present the following lifting property for Lie group actions, which will be very useful later.

\begin{thm}[{\cite[Theorem 5.1]{arone}}]\label{proposition lifting actions}
Let $G\begin{tikzcd}[ampersand replacement=\&, cramped, sep=small] \& \arrow[loop left] \phantom{X} \end{tikzcd} \hspace{-8pt} M$ be a smooth Lie group action and $S\subset M$ be a $G$-invariant, closed submanifold. Then the $G$-action lifts to a smooth action $G\begin{tikzcd}[ampersand replacement=\&, cramped, sep=small] \& \arrow[loop left] \phantom{X} \end{tikzcd} \hspace{-8pt} \blup{S}{M}$, with respect to which the blow-down projection $\pi\colon \blup{S}{M}\to M$ becomes $G$-equivariant.
\end{thm}

Notice that this result allows one to also lift complete vector fields that are tangent to $S$, since such fields induce an $\mathbb{R}$-action on $M$. The proof in \cite{arone}, however, also applies locally. It relies on the local properties of the blow-up construction (\cite[Theorem 4.1]{arone}) and on a general fact for blow-ups of cartesian products (\cite[Proposition 4.3]{arone}). This strategy hence also works when the vector fields are only locally defined, and not necessarily complete. That is, we can lift the local flows of general local vector fields, provided they are tangent to $S$ (see also \cite[Section~2]{arone}).

Now let us pass to the applications of this technique for Riemannian foliations. In \cite{desmolino}, successive blow-ups are used to desingularize $\overline{\fol}$, starting from a regular Riemannian foliation $\fol$. In that setting, $\blup{\Sg}{\fol}$ is constructed by explicitly inducing a foliated atlas for it, from the one for $\fol$. In contrast, the approach in \cite{desmarcos} for a singular Riemannian foliation $\fol$ is based on equifocality and adapts the construction in \cite[Section 2.9]{duistermaat} for blow-ups of proper Lie group actions. In what follows, we describe the latter construction by extending $T\fol$ across the exceptional divisor. In detail, let $(M,\fol,\metric)$ be a complete singular Riemannian foliation and let $\Sg$ be its minimal locus. We assume the depth is positive on the connected component under consideration. In order to simplify the desingularization process, we will blow up all connected components of $\Sg$ simultaneously, even when their dimensions differ. Note, however, that the process may as well be subdivided into intermediate blow-ups, component by component, if needed. The first ingredient we need is a suitable neighborhood $U\supset \Sg$. When $\Sg$ is compact, one can just take $U$ to be an $\varepsilon$-tubular neighborhood of $\Sg$, narrow enough so that the neighborhoods $\tub_{\varepsilon}(\Sg_{\alpha})$ of its components are distinguished and pairwise disjoint. The particular geometry of $\fol$, however, allows one to blow $\Sg$ up even when it is not compact, ultimately, thanks to the fact that closed leaves of Riemannian foliations always admit tubular neighborhoods \cite[Proposition 16]{radeschislice}. To do that, one needs to replace the tubular neighborhood with a saturated neighborhood $U$ (not necessarily tubular) with similar properties --- for instance, take
\[U=\bigcup_{\overline{L}\subset\Sigma}\operatorname{Tub}_{\varepsilon(\overline{L})}\left(\overline{L}\right).\]

If $X\in\mathfrak{X}(\fol)$, then it is tangent to $\Sg$ and hence lifts to a smooth vector field $\blup{\Sg}{X}$ on $\blup{\Sg}{M}$, as explained above. For each open set $\hat{U}\subset\blup{\Sg}{M}$ we define
\[
    \blup{\Sg}{\mathfrak{X}(\fol)}(\hat{U})
    =\bigl\{Y\in\mathfrak{X}(\hat{U})\mid
      \dif\pi_z(Y_z)\in T_{\pi(z)}\fol
      \text{ for all }z\in \hat{U}\setminus E\bigr\}.
\]
In particular, this family contains the individual lifts of fields tangent to $\fol$. It naturally spans a smooth singular distribution $\blup{\Sg}{T\fol}$ on $\blup{\Sg}{M}$, defined by evaluating the local fields at each point.

\begin{prop} \label{sing foliation construction}
    The singular distribution $\blup{\Sg}{T\fol}$ is integrable to a singular foliation $\blup{\Sg}{\fol}$ of $\blup{\Sg}{M}$, such that $\pi\colon\blup{\Sg}{M}\to M$ is a foliated map and $\pi|_{(\blup{\Sg}{M})\setminus E}$ is a foliated diffeomorphism onto $M\setminus\Sigma$. Moreover, $\operatorname{depth}(\blup{\Sg}{\fol})=\operatorname{depth}(\fol)-1$.
\end{prop}

\begin{proof}
    By the Stefan--Sussmann theorem (see, e.g., \cite[Theorem 4]{lavau}), it is sufficient to verify that $\blup{\Sg}{T\fol}$ is invariant under the flows of the fields in $\blup{\Sg}{\mathfrak{X}(\fol)}$. Locally near $E$, write the blow-down in signed radial coordinates as $(x,v,t)\mapsto \exp_x^\perp(tv)$, with $v$ unit and $(x,v,t)\sim(x,-v,-t)$. The distance cylinders are locally saturated by the homothetic lemma \cite[Proposition 2.7]{desmarcos}. Thus every field in the family annihilates $t$ off $E$, and hence everywhere in the signed radial chart. Its local flow preserves $E$ and is leafwise off $E$. By the defining condition on the family, the flow preserves the family and the distribution it spans.

    Positive homotheties identify the angular foliations on the distance cylinders. The normal reflection is also foliate. Indeed, the normal projection takes plaques in a small tube to plaques in $\Sg$ \cite[Proposition 2.13]{desmarcos}, and along each such plaque the reflection agrees with the reflection about its projected plaque \cite[Remark~2.1]{lytchakthorbergsson}. The product of the angular foliation with the foliation by points of the signed radial interval therefore descends under $(x,v,t)\sim(x,-v,-t)$ and agrees with the transported foliation off $E$. In these coordinates the angular tangent space at a fixed $(x,v)$ is independent of $t$. By continuity along this fixed radius, the fields in $\blup{\Sg}{\mathfrak{X}(\fol)}$ are tangent to the product at $t=0$ as well. Conversely, the local tangent fields of the product satisfy the defining condition of the family. This identifies the foliation at $E$ with the angular extension used in \cite[Propositions 3.3--3.4]{desmarcos}.

    Consequently, the leaf through $(x,[v])\in E$ has the same dimension as the leaf through $\exp_x^\perp(tv)$ for small $t\ne0$. The dimensions occurring after the blow-up are exactly those occurring outside $\Sg$ before the blow-up, which proves the assertion about depth. Finally, on $E$ the map $\pi$ is induced by the normal projection of the angular model, so it maps plaques to plaques of $\fol|_\Sg$. Outside $E$ the claim follows from the foliated diffeomorphism.
\end{proof}

The singular foliation $\blup{\Sg}{\fol}$ integrating $\blup{\Sg}{T\fol}$ is the \textit{blow-up of $\fol$}.

\begin{obs}\label{obs foliaton along exceptional divisor}
Along the exceptional divisor, $\blup{\Sg}{\fol}$ admits a very natural geometric interpretation, as we now describe. Since $E = \mathbb{P}(\nu\Sg)$, we can view it as the quotient of a small distance cylinder $C_r(\Sg)$ (the sphere bundle of radius $r$ in $\nu\Sg$) by the antipodal involution: $E \cong C_r(\Sg)/\mathbb{Z}_2$ (this description is local when there is no uniform tubular radius). The homotheties identify the angular foliation at different radii, and the normal reflection ensures that it descends to the projective quotient. The restriction of an individual lifted field to $E$ depends only on the field along $\Sg$ and its first normal derivative and such restrictions can span a smaller distribution than the tangent distribution of the full angular foliation; this is why one cannot define $\blup{\Sg}{\mathfrak{X}(\fol)}$  as simply the lift of $\mathfrak{X}(\fol)$.
\end{obs}

In order to promote $\blup{\Sg}{\fol}$ to a singular Riemannian foliation, one must equip the blow-up $\blup{\Sg}{M}$ with a suitable bundle-like metric. Notice that the naive pullback of the original metric $\metric$ is insufficient, as $\pi^*\metric$ degenerates along $E$. In \cite{desmarcos}, Alexandrino desingularizes the metric by systematically inflating the collapsing directions. The construction is a little bit convoluted, as it requires a sequence of similar steps to be performed locally, around individual leaves and $\Sigma$, but we can synthesize the construction at the $\Sigma$ level as follows: 

\begin{itemize}
    \item \textbf{Local Homogenization.}
    In a saturated tube $U$ about $\Sg$, let $\metric_0$ be Alexandrino's auxiliary adapted metric. At $x\in\Sg$, the two orthogonal summands in its construction are $T_xL_x$ and $(T_xL_x)^{\perp_{\metric}}$. On the first summand the projection to the center acts as the identity, and on the second the $\metric$-orthogonal projection in the construction also acts as the identity. Thus, for $x\in\Sg$,
    \begin{equation}\label{auxiliary metric on center}
    (\metric_0)_x=\metric_x
    \end{equation}
    (on $T_xM$). These are the slice and projection identifications in
    \cite[Propositions~2.17 and~3.1]{desmarcos}.

    \item \textbf{Radial Normalization.}
    Let $\rho$ denote the distance to $\Sg$ in the tubular coordinates. Over $U\setminus\Sg$, let $\mathcal V$ be tangent to the spherical fibers of the distance cylinders $C_\rho(\Sg)$ and let $\mathcal H$ be its $\metric_0$-orthogonal complement. With respect to the splitting $\mathcal H\oplus\mathcal V$, the rescaled metric is
    \begin{equation}\label{angular rescaling metric}
    \tilde{\metric}(X,Y)=\metric_0(X^{\mathcal H},Y^{\mathcal H})+\frac{1}{\rho^2}\metric_0(X^{\mathcal V},Y^{\mathcal V}).\end{equation}
    Its pullback off $E$ extends to a smooth adapted metric $\hat{\metric}$ on $\pi^{-1}(U)$. The restriction $\pi|_E$ is a Riemannian submersion onto $(\Sg,\metric|_\Sg)$ \cite[Equation~(3.1) and Proposition~3.4]{desmarcos}.

    \item \textbf{Foliated Gluing.}
    Finally, to obtain a global, $\blup{\Sg}{\fol}$-bundle-like metric $\blup{\Sg}{\metric}$ on $\blup{\Sg}{M}$, the local metric $\tilde{\metric}$ is glued to the original metric $\metric$ on $M \setminus U$. To preserve the bundle-like property, the gluing must be performed using a basic partition of unity by bump functions that depend solely on $\rho$ in the tube, so it is constant on each distance cylinder \cite[p.~412]{desmarcos}.
\end{itemize}

Let us summarize this in the following:

\begin{prop}\label{def blup} Let $(M,\fol,\metric)$ be a complete singular Riemannian foliation. Then there exists a bundle-like metric $\blup{\Sg}{\metric}$ making $\blup{\Sg}{\fol}$ into a singular Riemannian foliation. With respect to it, $\pi|_E$ is a Riemannian submersion onto $\Sigma$, and $\pi|_{(\blup{\Sg}{M})\setminus \pi^{-1}(U)}$ is an isometry onto $M\setminus U$.
\end{prop}

Since $\operatorname{depth}(\blup{\Sg}{\fol})=\operatorname{depth}(\fol)-1$, one can repeatedly apply this blow-up process until the depth drops to $0$, thus obtaining a regular foliation, i.e., desingularizing $\fol$. In more detail, let us denote the minimal locus of $\fol$ by $\Sigma_0$ and $\blup{1}{M}:=\blup{\Sg_0}{M}$. For $k\geq 1$ we define by induction $\blup{k+1}{M}:=\blup{\Sigma_{k}}{\blup{k}{M}}$, so that $\pi_{k+1}\colon \blup{k+1}{M} \rar \blup{k}{M}$ is the blow-down map along the minimal locus $\Sg_k$ of $\blup{k}{\fol}$. Similarly, we apply this induction to the foliation and the metric, obtaining a sequence $(\blup{k}{M},\blup{k}{\fol},\blup{k}{\metric})$. This process ends (becomes trivial) since $\operatorname{depth}(\fol)$ is finite and the blow up construction is trivial for regular foliations. This final stage, called the \emph{desingularization of $(M,\fol,\metric)$} , will be denoted simply by $(\blup{}{M},\blup{}{\fol},\blup{}{\metric})$, with projection $\Pi:=\pi_1\circ\cdots\circ\pi_{\operatorname{depth}(\fol)}\colon \blup{}{M} \rar M$.

\section{Dynamics of the blown-up foliation}
In \cite{desmolino}, Molino showed that, for the closure of a regular Riemannian foliation, the Molino sheaf of the blow-up foliation \( \blup{\Sigma}{\fol} \) coincides with the blow-up sheaf of \( \fol \). In this section, we give the analogous comparison for singular Riemannian foliations with smooth Molino sheaf, using Alexandrino's blow-up reviewed in Section~\ref{alexandrino blow}.

Let $U \subset M$ be an open set and $X \in \mathfrak{L}(\fol|_U)$ a local foliate vector field. Its local flow preserves $\fol|_U$, and therefore preserves all of its strata, so in particular $X$ is tangent to $\Sg \cap U$. By our previously established lifting procedure, its local flow lifts to $\pi^{-1}(U)$, defining a smooth lifted vector field $\blup{\Sg}{X}$ on $\pi^{-1}(U)$.

\begin{prop}\label{propositon: blow up of foliate}
If $X \in \mathfrak{L}(\fol|_U)$, then $\blup{\Sg}{X}\in \mathfrak{L}(\blup{\Sg}{\fol}|_{\pi^{-1}(U)})$. Moreover, the lifting map $\blup{\Sg}{}\colon \mathfrak{L}(\fol|_U) \to \mathfrak{L}(\blup{\Sg}{\fol}|_{\pi^{-1}(U)})$ is an injective Lie algebra homomorphism and satisfies $\blup{\Sg}(\mathfrak{X}(\fol|_U))\subset \mathfrak{X}(\blup{\Sg}{\fol}|_{\pi^{-1}(U)})$.
\end{prop}

\begin{proof}
    The lift of a tangent field belongs to $\blup{\Sg}{\mathfrak{X}(\fol)}(\pi^{-1}(U))$, since it is tangent off $E$. For a general $X\in\mathfrak{L}(\fol|_U)$, take any local tangent field $Z$ of $\blup{\Sg}{\fol}$. Outside the exceptional divisor $E$, the field $\blup{\Sg}{X}$ is foliate by the foliated diffeomorphism $\pi$. Thus $[\blup{\Sg}{X},Z]$ is tangent off $E$ and satisfies the defining condition of $\blup{\Sg}{\mathfrak{X}(\fol)}$ on its domain. This shows that $\blup{\Sg}{X}$ is foliate. Linearity and the bracket identity $[\blup{\Sg}{X},\blup{\Sg}{Y}]=\blup{\Sg}{[X,Y]}$ follow from $\pi$-relatedness off $E$ and smoothness. Injectivity follows from the diffeomorphism off $E$ and the density of $U\setminus\Sg$ in $U$.
\end{proof}

It follows that, for a local transverse field $\overline{X}\in \mathfrak{l}(\fol_U)$, we can define its blow-up $\blup{\Sg}{\overline{X}} \in \mathfrak{l}(\blup{\Sg}{\fol}|_{\pi^{-1}(U)})$ by $\overline{\blup{\Sg}{X}}$. In fact, if $X'\in\mathfrak{L}(\fol|_U)$ is another representative, then $X'=X+T$ for some $T\in\mathfrak{X}(\fol_U)$. By linearity,
\[\blup{\Sg}{X'}=\blup{\Sg}{X+T}=\blup{\Sg}{X}+\blup{\Sg}{T}.\]
Since Proposition \ref{propositon: blow up of foliate} guarantees $\blup{\Sg}{T} \in \mathfrak{X}(\blup{\Sg}{\fol}|_{\pi^{-1}(U)})$, we get $\overline{\blup{\Sg}{X'}}=\overline{\blup{\Sg}{X}}$, proving the operation is well-defined. It is injective because if $\blup{\Sg}{X}$ is tangent, its projection under the surjective foliated map $\pi$ is tangent. For a basic function $f$, one has $\blup{\Sg}{(fX)}=(f\circ\pi)\blup{\Sg}{X}$ by equality off $E$ and smoothness.

Covariant objects lift straightforwardly via
\[
    \pi^*\colon \Omega^k(\f|_U)\to \Omega^k(\blup{\Sg}{\f}|_{\pi^{-1}(U)}).
\]
More generally, local $\f$-basic covariant tensor fields pull back via $\pi$ to local $\blup{\Sg}{\f}$-basic covariant tensor fields. The basicness identities can be checked off $E$ against arbitrary tangent fields upstairs and then extended by smoothness. Each $\pi^*$ is also injective by continuity, since $\pi$ restricts to a diffeomorphism $\pi^{-1}(U)\setminus E \to U\setminus (\Sg \cap U)$. In particular, there is a natural inclusion $C^{\infty}(\f|_U)\to C^\infty(\blup{\Sg}{\f}|_{\pi^{-1}(U)})$, and we can state the following proposition.

\begin{prop} \label{blup transverse} 
    For any open set $U \subset M$, the blow-up of foliate vector fields yields an injective, $C^{\infty}(\f|_U)$-linear Lie algebra homomorphism $\blup{\Sg}{}\colon\mathfrak{l}(\fol|_U)\to\mathfrak{l}(\blup{\Sg}{\fol}|_{\pi^{-1}(U)})$.
\end{prop}

We define the blow-up of $\mathscr{C}_{\fol}$ as the inverse image
\[\blup{\Sg}{(\mathscr{C}_{\fol})}:= \pi^{-1}(\mathscr{C}_{\fol}).\]
Recall that this is the sheafification of the presheaf of ``generalized germs'': given an open set $\hat{U}\subset \blup{\Sg}{M}$, the presheaf consists of equivalence classes of sections $\overline{Y}$ of $\mathscr{C}_{\fol}$  on open sets $U\subset M$ containing $\pi(\hat{U})$, under the relation that identifies $\overline{Y}\in \mathscr{C}_{\fol}(U)$ with $\overline{Y}'\in \mathscr{C}_{\fol}(U')$ whenever they agree on a smaller open neighborhood of $\pi(\hat{U})$ in $U\cap U'$.

\begin{thm} \label{blup sheaf} Let $\fol$ be a singular Riemannian foliation with smooth Molino sheaf, of a compact manifold $M$ . Then the lifting of transverse fields induces an isomorphism
\[\mathscr{C}_{\blup{\Sg}{\fol}} \cong \blup{\Sg}{\mathscr{C}_{\fol}}.\]
In particular, the Molino sheaf of $\blup{\Sg}{\fol}$ is smooth.
\end{thm}

    \begin{proof}
    Since $\pi$ restricts to a foliated diffeomorphism $\pi^{-1}(\Sigma_{\reg})\to \Sigma_{\reg}$, it is clear that $\mathscr{C}_{\blup{\Sg}{\fol}}$ and $\blup{\Sg}{(\mathscr{C}_{\fol})}$ are isomorphic when restricted to $\pi^{-1}(\Sigma_{\reg})$. The pullback by $\pi$ identifies their sections through the induced conjugacy of the holonomy pseudogroups and their closures \cite[Appendix~D]{molino}. The result will therefore follow once we establish that this identification extends by lifting to the entire regular locus $(\blup{\Sigma}{M})_{\mathrm{reg}}$  of $\blup{\Sg}{\fol}$. Indeed, both sheaves are locally constant, and the intersection of a sufficiently small ball with $(\blup{\Sigma}{M})_{\mathrm{reg}}$ is connected because the singular strata have codimension at least two. An isomorphism over $(\blup{\Sigma}{M})_{\mathrm{reg}}$ thus extends uniquely across the singular locus. Near a regular point of $E$, however, $\pi^{-1}(\Sigma_{\reg})$ has two local components, and the identifications on the two sides must agree in trivializations of the sheaves across $E$.

    Near such a point, lift smooth foliate representatives of a local basis of $\mathscr{C}_{\fol}$ and choose a regular foliated chart $\hat{U}$ with connected transversal on which the lifts are defined and $\mathscr{C}_{\blup{\Sg}{\fol}}$ is constant. The projected lifts satisfy the Killing equation on the dense open subset corresponding to $\hat{U}\cap \pi^{-1}(\Sigma_{\reg})$ and hence on the whole transversal. Their germs at a point of $\hat{U}\cap \pi^{-1}(\Sigma_{\reg})$ form a basis of the upstairs Molino stalk and extend uniquely as sections over $\hat{U}$. On the connected transversal, the projected lifts and these sections are Killing fields that agree near the chosen point, so they agree everywhere. This gives the identification by lifting on $(\blup{\Sigma}{M})_{\mathrm{reg}}$.
\end{proof}

\begin{cor}\label{blup kil is kil} If $\fol$ is a singular Killing foliation of a compact manifold $M$, then $\blup{\Sg}{\fol}$ is a singular Killing foliation, and the structural algebras $\mathfrak{g}_\fol$ and $\mathfrak{g}_{\blup{\Sg}{\fol}}$ are isomorphic.
\end{cor}

\begin{proof}
The first claim follows from Theorem \ref{blup sheaf} and the fact that inverse images of constant sheaves are constant. For the second one, using Theorem \ref{blup sheaf}, for any regular point $x\in M$ and its preimage $\hat{x}\in \blup{\Sg}{M}$,
\[\mathfrak{g}_\fol \cong \mathrm{
stalk}_x(\mathscr{C}_{\fol})^{\mathrm{op}} \cong \mathrm{stalk}_{\hat{x}}(\blup{\Sg}{\mathscr{C}_{\fol}})^{\mathrm{op}} \cong \mathrm{stalk}_{\hat{x}}(\mathscr{C}_{\blup{\Sg}{\fol}})^{\mathrm{op}} \cong \mathfrak{g}_{\blup{\Sg}{\fol}}.\qedhere\]
\end{proof}

Applying this repeatedly gives the following corollary.

\begin{cor}\label{corollary: desing killing is killing} If $\fol$ is a singular Killing foliation of a compact manifold $M$, then $\blup{}{\fol}$ is a (regular) Killing foliation, and $\mathfrak{g}_\fol\cong \mathfrak{g}_{\blup{}{\fol}}$.
\end{cor}

On a compact manifold with finite fundamental group and nonzero Euler characteristic, every regular Riemannian foliation is closed \cite[Theorem~F]{poscara}. To prove Theorem~\ref{thmB}, we use a version of this criterion suited to the desingularization.

Let us first understand what happens with the Euler characteristic in the blow-up process. For that, we recall a result of Gitler \cite{gitler}, who computed the singular cohomology (with $\mathbb{Z}_{2}$ coefficients) of the blow-up  $\blup{\Sigma}{M}$. Later we will generalize it to the foliated setting. Suppose $\Sg$ is a closed submanifold of $M$. Gitler proved that
\begin{equation}\label{gitler equation}
    H^i(\blup{\Sg}{M},\mathbb{Z}_2) \cong \pi^*(H^i(M,\mathbb{Z}_2)) \oplus \dfrac{H^i(E,\mathbb{Z}_2)}{(\pi|_E)^*(H^i(\Sg,\mathbb{Z}_2))}.
\end{equation}
Recall that, for a compact $M$, the Euler characteristic can be computed as
\[\chi(M)=\sum_n(-1)^n\dim H_n(M;F),\]
for an arbitrary coefficient field $F$ \cite[Theorem 2.44]{hatcher}. We may now easily compute the Euler characteristic of \(\blup{\Sg}{M}\). The pullback $\pi^*$ is injective by the Gysin map and Poincar\'e duality with $\mathbb Z_2$ coefficients \cite{gitler}. The projective bundle theorem also gives injectivity of $(\pi|_E)^*$ with these coefficients.
Equation~\eqref{gitler equation} therefore gives
$$\dim H^i(\blup{\Sg}{M},\mathbb{Z}_2) + \dim H^i(\Sg,\mathbb{Z}_2)   =  \dim H^i(E,\mathbb{Z}_2) + \dim H^i(M,\mathbb{Z}_2).$$
The alternate sum on $i$ then leads to
\[\chi(\blup{\Sg}{M})+\chi(\Sg) = \chi(E) +\chi(M).\]
Furthermore, from the fiber bundle property for Euler characteristic we know that $\chi(E)=\chi(\Sg)\chi(\RR\mathbb{P}^{k-1})$, where $k=\codim(\Sg)$. Thus
\[\chi(\blup{\Sg}{M})=\chi(\Sg)(\chi(\mathbb{RP}^{k-1})-1)+\chi(M),\]
and so
\begin{equation} \label{euler charac}
  \chi(\blup{\Sg}{M}) = \begin{cases}
    \chi(M) & \text{if} \ k \ \text{is odd},  \\
   
    \chi(M) - \chi(\Sg)  & \text{if} \ k \ \text{is even}.
\end{cases} 
\end{equation}
Using this, we observe that, for every blow-up along a minimal stratum $\Sg$ (here we will use the refined ``stratum by stratum'' process, as we discussed in Section \ref{alexandrino blow}), supposing $\chi(M)\neq 0$ we have:
\begin{enumerate}
    \item If $\Sg$ has odd codimension, then $\chi(\blup{\Sg}{M}) \neq 0$.
    \item If $\Sg$ has even codimension and $\chi(M) \neq \chi(\Sg)$, then $\chi(\blup{\Sg}{M})\neq 0$.
\end{enumerate}

In order to desingularize $\fol$ we must do a finite number of blow-ups, and in each step both cases may occur. So, proceeding inductively, we are led to consider the collection $\Sigma_1,\dots,\Sigma_\ell$ of all the even-codimensional minimal strata appearing in the blow-up process (notice that not all of them are strata of $\fol$, but rather of its blow-ups) and the number
\[\mathscr{L}:=\sum_{j=1}^\ell \chi(\Sigma_{j}).\]

For $\blup{}{\fol}$, the relevant group is the image of the fundamental group of $\Pi^{-1}(\Sigma_{\reg})$ in $\pi_1(\blup{}{M})$. The following criterion extends the regular result in the Killing case
\cite[Theorem~F]{poscara}.

\begin{prop}\label{regular closing criterion}
Let $\fol$ be a regular Killing foliation of a compact connected manifold $M$ with $\chi(M)\ne0$. Suppose there is a nonempty connected open union of leaf closures $U$ such that $\pi_1(U)$ has finite image in $\pi_1(M)$. Then $\fol$ is closed.
\end{prop}

\begin{proof}
By \cite[Proposition~3.3]{poscara}, $\fol$ admits a transverse Killing field whose zero set is the union of its closed leaves. Its foliate representative orthogonal to the leaves has the same zero set. Since $\chi(M)\ne 0$, the Poincar\'e--Hopf theorem implies that this set is nonempty, so we can fix a closed leaf $L\in \fol$.

Choose closed Riemannian foliations $\mathcal G_i$ approaching $\fol$, as in \cite[Theorem~B]{poscara}. The leaves of each $\mathcal G_i$ lie within the leaf closures of $\fol$, so $U$ is $\mathcal G_i$-saturated. Moreover, $L$ remains a leaf of $\mathcal G_i$, since it is connected and the deformations preserve the leaf dimension. Let us fix a leaf with trivial holonomy $L_i\in \mathcal G_i$ within $U$, for each $i$.

Haefliger's classifying map $M\to B(M/\!/\mathcal G_i)$ is homotopy equivalent to a locally trivial fibration whose fiber is the common holonomy cover of the leaves, thus diffeomorphic to $L_i$ (see \cite[Section~9]{poscara} for details). Over the point represented by $L_i$, the local product model identifies the fiber map into $M$ with the inclusion $L_i\hookrightarrow M$ \cite[Corollaries~3.1.3--3.1.5 and Section~4.1]{haefligerclass}. The homotopy sequence therefore gives
\[\pi_1(L_i)\longrightarrow\pi_1(M)\longrightarrow\pi_1^{\mathrm{orb}}(M/\!/\mathcal G_i)\longrightarrow 1.\]
We fix a base point in $L$ for $\pi_1(M)$ and use connecting paths as needed. Since $L_i\subset U$, the first homomorphism factors through $\pi_1(U)$. Its image is finite by hypothesis, so the second homomorphism has finite kernel.

Take a finite triangulation of the underlying space of $M/\!/\mathcal G_i$ adapted to its orbifold charts \cite[Proposition~1.2.1 and Lemma~1.2.2]{moerdijkpronk_simplicial}. For each open simplex $\sigma$, choose a lift in a chart and let $H_\sigma$ be the common stabilizer of its points. The lift maps homeomorphically onto $\sigma$, and $H_\sigma$ fixes it pointwise. The Reeb model thus identifies the preimage of $\sigma$ under the leaf projection with $\sigma\times(L_i/H_\sigma)$, where $H_\sigma$ acts freely on $L_i$ by deck transformations.

The preimages of the skeleta give a finite filtration of $M$ by closed subsets. Each successive difference is the disjoint union of the preimages of the open simplices of the corresponding dimension. The long exact sequences in compactly supported cohomology give additivity of the compactly supported Euler characteristic $\chi_c$ along this filtration. Since $M$ is compact, $\chi_c(M)=\chi(M)$. An open simplex $\sigma$ is homeomorphic to $\mathbb R^{\dim\sigma}$, so its compactly supported cohomology is one-dimensional in degree $\dim\sigma$ and zero in all other degrees. The K\"unneth formula with compact supports, together with the compactness of $L_i/H_\sigma$, therefore gives
\[\begin{aligned}
\chi(M)&= \sum_\sigma \chi_c\bigl(\sigma\times(L_i/H_\sigma)\bigr) \\
  &= \sum_\sigma (-1)^{\dim\sigma}\chi(L_i/H_\sigma) \\
  &= \chi(L_i)\sum_\sigma\frac{(-1)^{\dim\sigma}}{|H_\sigma|}.
\end{aligned}\]
The last equality uses the covering $L_i\to L_i/H_\sigma$ of degree $|H_\sigma|$. Since $\chi(M)\ne 0$, it follows that $\chi(L_i)\ne 0$.

Now let $H_i$ be the holonomy group of $L$ in $\mathcal G_i$. Its Reeb neighborhood projects as $(L_i\times D)/H_i\to D/H_i$, where $D$ is a transverse disk centered at the point corresponding to $L$, and $L_i\to L$ is the holonomy covering. Over the classifying space of this chart, the fibration is represented by
\[EH_i\times_{H_i}(L_i\times D)\longrightarrow EH_i\times_{H_i}D.\]
Pulling back along $BH_i\to B(M/\!/\mathcal G_i)$, the map determined by the center of $D$, therefore gives $EH_i\times_{H_i}L_i\to BH_i$ (see \cite[Sections~4.1--4.2]{haefligerclass} for more details). The cohomological monodromy of this pullback is the action of $H_i$ by deck transformations on the cohomology of $L_i$. Since it is pulled back from the global fibration, the kernel $K_i$ of $H_i\to\pi_1^{\mathrm{orb}}(M/\!/\mathcal G_i)$ acts trivially on $H^*(L_i;\mathbb R)$. Transfer for the covering $L_i\to L_i/K_i$ then gives
\[H^*(L_i/K_i;\mathbb R)\cong H^*(L_i;\mathbb R)^{K_i}=H^*(L_i;\mathbb R)\]
\cite[Proposition~3G.1]{hatcher}. Hence $\chi(L_i)=|K_i|\chi(L_i/K_i)=|K_i|\chi(L_i)$. As $\chi(L_i)\ne 0$, we have $K_i=1$.

The projection $EH_i\times_{H_i}L_i\to L_i/H_i=L$ has contractible fiber $EH_i$. Under this homotopy equivalence, the map from the local total space to $M$ corresponds to the inclusion of $L$. By naturality of the homotopy sequences, the composite
\[\pi_1(L)\longrightarrow\pi_1(M)\longrightarrow\pi_1^{\mathrm{orb}}(M/\!/\mathcal G_i)\]
is the holonomy surjection $\pi_1(L)\to H_i$ followed by the injection we just proved. Thus the image of $\pi_1(L)$ in $\pi_1(M)$ surjects onto $H_i$ with finite kernel, and is therefore finite. Since the deformations preserve $L$ as the same submanifold of $M$, this single finite subgroup bounds the order of every $H_i$. If $\fol$ were nonclosed, these orders would tend to infinity by \cite[Lemma~4.3]{poscara}.
\end{proof}

\begin{thm}\label{closed leaves kill}
Let $\fol$ be a singular Riemannian foliation with smooth Molino sheaf on a compact connected manifold $M$ with finite fundamental group. Suppose that every singular stratum has codimension at least three. If $\chi(M)\ne\mathscr L$, then $\fol$ is closed.
\end{thm}

\begin{proof}
Applying \eqref{euler charac} successively gives $\chi(\blup{}{M})=\chi(M)-\mathscr L\ne0$. Since every singular stratum has codimension at least three, the inclusion $\Sigma_{\reg}\hookrightarrow M$ induces an isomorphism on fundamental groups. Since $\Pi$ restricts to a foliated diffeomorphism over $\Sigma_{\reg}$, the open set $\Pi^{-1}(\Sigma_{\reg})$, which of course is a union of leaf closures, has finite fundamental group.

Iterating Theorem~\ref{blup sheaf} gives $\mathscr C_{\blup{}{\fol}}\cong\Pi^{-1}\mathscr C_\fol$. Its monodromy therefore factors through the finite group $\pi_1(M)$, so we can pass to the finite covering of $\blup{}{M}$ associated with the kernel of this monodromy. By \cite[Proposition~3.5]{poscara}, the lifted regular foliation is Killing. The covering manifold has nonzero Euler characteristic, and each connected component of the preimage of $\Pi^{-1}(\Sigma_{\reg})$ is again an open union of leaf closures with finite fundamental group. Proposition~\ref{regular closing criterion} thus gives closedness of the lifted foliation, which descends through the finite covering.

Since $\blup{}{\fol}$ is regular and closed, its Molino sheaf has zero stalks. The iterated comparison $\mathscr C_{\blup{}{\fol}}\cong\Pi^{-1}\mathscr C_\fol$ and surjectivity of $\Pi$ give $\mathscr C_\fol=0$. Hence $\fol$ is closed.
\end{proof}

\begin{cor}\label{cor odd dimensions}\label{corollary: singular riemannian closed}
Let $\fol$ be a singular Riemannian foliation with smooth Molino sheaf on a compact connected manifold $M$ with finite fundamental group and $\chi(M)\ne0$. If all singular strata of $\fol$ have odd codimension, then $\fol$ is closed.
\end{cor}

\begin{proof}
Singular strata have codimension at least two, so the odd-codimensional strata have codimension at least three. The angular product model in Proposition~\ref{sing foliation construction} shows that the codimension of a stratum meeting the exceptional divisor agrees with its codimension off the divisor. Thus all successive centers have odd codimension and
$\mathscr L=0$. Theorem~\ref{closed leaves kill} applies.
\end{proof}

\begin{cor}\label{local closed leaf criterion}
Let $\fol$ be a singular Killing foliation on a compact connected manifold $M$. Suppose there is a closed leaf $L\subset\Sigma_{\reg}$ with $\chi(L)\ne0$ such that the image of $\pi_1(L)$ in $\pi_1(\Sigma_{\reg})$ is finite. Then $\fol$ is closed.
\end{cor}

\begin{proof}
First suppose that $\fol$ is regular, so $\Sigma_{\reg}=M$. Choose closed deformations $\mathcal G_i$ approaching $\fol$ as in \cite[Theorem~B]{poscara}. The closed leaf $L$ is a leaf of every $\mathcal G_i$. If $H_i$ is its holonomy group in $\mathcal G_i$, its holonomy covering $\hat L_i$ satisfies
\[\chi(\hat L_i)=|H_i|\chi(L)\neq 0.\]
The monodromy and transfer argument in the proof of Proposition~\ref{regular closing criterion} shows that $H_i$ injects
into $\pi_1^{\mathrm{orb}}(M/\!/\mathcal G_i)$ and that the image of $\pi_1(L)$ in $\pi_1(M)$ maps onto $H_i$. Consequently all the orders $|H_i|$ are bounded by the order of this fixed finite image. If $\fol$ were nonclosed, these orders would tend to infinity
\cite[Lemma~4.3]{poscara}.

For a singular $\fol$, the lift $\Pi^{-1}(L)$ is a closed leaf of the regular Killing foliation $\blup{}{\fol}$ and is diffeomorphic to $L$. Its fundamental group has finite image in $\pi_1(\blup{}{M})$ because the inclusion factors through $\Pi^{-1}(\Sigma_{\reg})$, which is diffeomorphic to $\Sigma_{\reg}$. The regular case gives closedness of $\blup{}{\fol}$, and the iterated Molino sheaf comparison gives closedness of $\fol$.
\end{proof}

\section{Blown-up leaf spaces}

For a singular Riemannian foliation with compact leaves, Alexandrino proved that its leaf space is a Gromov--Hausdorff limit of Riemannian orbifolds obtained by desingularization \cite[Theorem~1.5 and Corollary~1.6]{desmarcos}. We strengthen the comparison for a single blow-up in \cite[Proposition~4.1]{desmarcos} by controlling the distances between points in the ambient manifolds. This estimate also applies to foliations with nonclosed leaves and gives a comparison of the spaces of leaf closures.

\begin{prop}\label{blow-down metric approximation}
Let $\fol$ be a singular Riemannian foliation on a compact connected manifold $M$ with adapted metric $\metric$. Let $\Sg$ be its minimal locus and let $\pi\colon\blup{\Sg}{M}\to M$ be the blow-down map. There are adapted metrics $\blup{\Sg}{\metric}_r$ on $\blup{\Sg}{M}$ such that
\begin{equation}\label{ambient blow-down distortion}
\sup_{\hat x,\hat y\in\blup{\Sg}{M}}\left|d_{\blup{\Sg}{\metric}_r}(\hat x,\hat y)-d_{\metric}\bigl(\pi(\hat x),\pi(\hat y)\bigr)\right|\longrightarrow 0.
\end{equation}
as $r\to0$. The induced map on spaces of leaf closures is surjective, and its distortion also tends to zero.
\end{prop}

\begin{proof}
Fix $\metric_0$ and $\hat{\metric}$ on a tube $U$ and its inverse image, respectively, as in Section~\ref{alexandrino blow}. In the gluing step, use a basic cutoff supported in $\pi^{-1}(\tub_r(\Sg))$ and equal to one near $E$, so the resulting metric $\blup{\Sg}{\metric}_r$ equals $\pi^*\metric$ outside $\pi^{-1}(\tub_r(\Sg))$.

Choosing
\[\eta_r:=\sup\bigl|(\metric_0)_x(v,v)-1\bigr|,\]
where the supremum ranges over all $x\in\overline{\tub_r(\Sg)}$ and all $v$ with $\metric_x(v,v)=1$, we have $\eta_r\to0$ as $r\to0$ by compactness of $\Sg$ and \eqref{auxiliary metric on center}. The definition gives
\[(1-\eta_r)\metric\leq\metric_0\leq(1+\eta_r)\metric\]
on $\tub_r(\Sg)$. For small radii, the rescaling of $\mathcal V$ in \eqref{angular rescaling metric} increases the metric. Thus $\hat{\metric}\geq\pi^*\metric_0$, with the inequality holding along $E$ by continuity. The gluing therefore gives
\[\blup{\Sg}{\metric}_r\geq(1-\eta_r)\pi^*\metric \quad\text{and}\quad d_{\blup{\Sg}{\metric}_r}(\hat x,\hat y)
\geq\sqrt{1-\eta_r}\,d_{\metric}\bigl(\pi(\hat x),\pi(\hat y)\bigr).\]

For the upper bound, the same gluing gives a fixed Riemannian metric $h$ on $\blup{\Sg}{M}$ dominating $\blup{\Sg}{\metric}_r$ and $\pi^*\metric$ for all sufficiently small $r$. Given $\delta>0$, choose a finite $\delta$-net for $d_h$ with points outside $E$. For each pair of net points $\hat a,\hat b$, choose a curve from $\pi(\hat a)$ to $\pi(\hat b)$ of length less than $d_{\metric}(\pi(\hat a),\pi(\hat b))+\delta$ that avoids $\Sg$. Such curves are obtained by perturbing minimizing geodesics relative to their endpoints, since the components of $\Sg$ have codimension at least two. For small enough $r$, all these curves avoid the tube of radius $r$. Their lifts have the same lengths, so
\[d_{\blup{\Sg}{\metric}_r}(\hat a,\hat b)\leq d_{\metric}\bigl(\pi(\hat a),\pi(\hat b)\bigr)+\delta.\]
For arbitrary $\hat x,\hat y$, the triangle inequality through the net points and the domination by $h$ give
\[d_{\blup{\Sg}{\metric}_r}(\hat x,\hat y)\leq d_{\metric}\bigl(\pi(\hat x),\pi(\hat y)\bigr)+5\delta.\]
Together with the lower bound, this proves \eqref{ambient blow-down distortion}.

The blow-down maps each leaf $\hat L$ onto the leaf through its image. Indeed, tangent fields generate the leaves downstairs, and their flows lift to tangent flows by Proposition~\ref{propositon: blow up of foliate}. Compactness then gives
\[\pi(\overline{\hat L})=\overline{\pi(\hat L)}.\]
Since the quotient metrics are given by distances between leaf closures, taking infima over pairs of points in two such closures bounds the distortion of the induced quotient map by the left-hand side of \eqref{ambient blow-down distortion}. Surjectivity follows from surjectivity of $\pi$.
\end{proof}

For a regular Killing foliation $\fol$ of codimension $d$ on a compact connected manifold $M$ with a bundle-like metric, Caramello and Neubauer obtain closed Riemannian foliations $\mathcal G_i$ of codimension $d$ with bundle-like metrics for which
\begin{equation}\label{regular killing gh approximation}
M/\mathcal G_i \xrightarrow{\mathrm{GH}} M/\overline{\fol}.
\end{equation}
The limit uses the quotient metric of the given bundle-like metric \cite[Theorem~5.2]{chicos}.

\begin{thm}\label{chicos def}
Let $\fol$ be a singular Killing foliation of codimension $d$ on a compact connected manifold $M$ with an adapted metric. Then there is a sequence of closed Riemannian foliations $\mathcal G_i$ of codimension $d$ on the desingularization $\blup{}{M}$, with bundle-like metrics for which
\[\blup{}{M}/\mathcal G_i \xrightarrow{\mathrm{GH}} M/\overline{\fol}.\]
In particular, the approximating quotients are $d$-dimensional Riemannian orbifolds.
\end{thm}

\begin{proof}
Fix $\varepsilon>0$. At each of the finitely many blow-up steps, choose the metric by Proposition~\ref{blow-down metric approximation} so that the sum of the successive Gromov--Hausdorff distances between the spaces of leaf closures is less than $\varepsilon/2$. This gives an adapted metric on the desingularization for which
\[d_{\mathrm{GH}}\!\left(\blup{}{M}/\overline{\blup{}{\fol}},\,M/\overline{\fol}\right)<\varepsilon/2.\]
The smooth projective blow-ups depend on their centers, and the blown-up foliations are determined by the blow-down maps as in Proposition~\ref{sing foliation construction}. Thus the same foliated desingularization can be used for every $\varepsilon$.

By Corollary~\ref{corollary: desing killing is killing}, $\blup{}{\fol}$ is a regular Killing foliation. Its codimension is $d$, since the maximal leaf dimension is preserved at each step. The desingularization is compact and is connected because the complement of the exceptional divisor at each step is connected and dense. Apply \eqref{regular killing gh approximation} with the bundle-like metric just chosen. This gives a closed Riemannian foliation $\mathcal G$ of codimension $d$ on $\blup{}{M}$ whose quotient is within $\varepsilon/2$ of $\blup{}{M}/\overline{\blup{}{\fol}}$. By the triangle inequality, its quotient is within $\varepsilon$ of $M/\overline{\fol}$. Taking $\varepsilon=1/i$ gives the desired sequence. Each $\blup{}{M}/\mathcal G_i$ is a Riemannian orbifold of dimension $d$ \cite[Theorem~2.15]{mrcun}.
\end{proof}

\section{Basic cohomology of the blown-up foliation}\label{section: bas cohom of blowup}

Let $(M,\fol)$ be a complete singular Riemannian foliation with minimal locus $\Sg$, and let $\pi\colon\blup{\Sg}{M}\to M$ be the blow-down map. Write $i_\Sg\colon\Sg\hookrightarrow M$ and $i_E\colon E\hookrightarrow\blup{\Sg}{M}$ for the inclusions. All cohomology groups in this section have real coefficients.

Let $U$ be a saturated foliated tubular neighborhood of $\Sg$. For $0\leq\lambda\leq1$, define
\[h_\lambda\bigl(\exp_x^\perp(w)\bigr)=\exp_x^\perp(\lambda w)\quad \text{and}\quad \hat h_\lambda([x,v,t])=[x,v,\lambda t],\]
where $w\in\nu_x\Sg$ and $[x,v,t]$ denotes the class of $(x,v,t)$ in the signed radial model of Proposition~\ref{sing foliation construction}. For $\lambda>0$, $h_\lambda$ is foliate by the extension of Lemma~\ref{homothetic lemma} to the minimal strata. At $\lambda=0$, the normal projection is foliate by \cite[Proposition~2.13]{desmarcos}. The formula for $\hat h_\lambda$ respects $(x,v,t)\sim(x,-v,-t)$ and maps each angular plaque at $t$ into the corresponding plaque at $\lambda t$. Moreover, $\pi\circ\hat h_\lambda=h_\lambda\circ\pi$.

We therefore get foliated retractions $\rho_\Sg\colon U\to\Sg$ and $\rho_E\colon\pi^{-1}(U)\to E$, given by
\[\exp_x^\perp(w)\mapsto x\quad \text{and}\quad [x,v,t]\mapsto[x,v,0],\]
respectively. By construction, the homotopies fix $\Sg$ and $E$, respectively, and satisfy
\[h_1=\id_U,\quad h_0=i_\Sg\circ\rho_\Sg,\quad \hat h_1=\id_{\pi^{-1}(U)}\quad\text{and}\quad\hat h_0=i_E\circ\rho_E,\]
where the inclusions in these identities take values in the respective tubes. By \cite[Corollary~2.5]{equicaramello}, the restriction from each tube to its center then induces an isomorphism in basic cohomology, with inverse $\rho_\Sg^*$ or $\rho_E^*$, respectively.

Define the relative basic complexes by
\[
\begin{aligned}
\Omega(\fol,\Sg)
  &=\ker\bigl(i_\Sg^*\colon\Omega(\fol)\to\Omega(\fol|_\Sg)\bigr),\\
\Omega(\blup{\Sg}{\fol},E)
  &=\ker\bigl(i_E^*\colon\Omega(\blup{\Sg}{\fol})
       \to\Omega(\blup{\Sg}{\fol}|_E)\bigr).
\end{aligned}
\]
Their cohomology groups are denoted by $\Ho^i(\fol,\Sg)$ and $\Ho^i(\blup{\Sg}{\fol},E)$. We define the relative complexes on $U$ and $\pi^{-1}(U)$ in the same way.

Put $V=M\setminus\Sg$ and choose a basic partition of unity $(\varphi,1-\varphi)$ subordinate to $\{U,V\}$. Thus $\operatorname{supp}\varphi\subset U$ and $\operatorname{supp}(1-\varphi)\subset V$. Write $\operatorname{ext}_W$ for extension by zero to $W$. For $\alpha\in\Omega(\fol|_\Sg)$ and $\beta\in\Omega(\blup{\Sg}{\fol}|_E)$, the support conditions make the following extensions smooth and basic:
\[
\begin{aligned}
i_\Sg^*\operatorname{ext}_M(\varphi\,\rho_\Sg^*\alpha)
  &=\alpha,\\
i_E^*\operatorname{ext}_{\blup{\Sg}{M}}
  \bigl((\pi^*\varphi)\,\rho_E^*\beta\bigr)
  &=\beta.
\end{aligned}
\]
Hence both global restriction maps are surjective on basic forms. On the tubes, $\rho_\Sg^*$ and $\rho_E^*$ are right inverses of the restrictions as cochain maps. The restrictions are quasi-isomorphisms by the homotopies above, so their kernels are acyclic:
\[\Ho^i(\fol|_U,\Sg)=0\quad\text{and}\quad \Ho^i(\blup{\Sg}{\fol}|_{\pi^{-1}(U)},E)=0\]
for every $i$.

\begin{prop}\label{relative blowup iso}
The blow-down induces an isomorphism
\[\pi^*\colon\Ho^i(\fol,\Sg)\longrightarrow\Ho^i(\blup{\Sg}{\fol},E)\]
in every degree.
\end{prop}

\begin{proof}
The identity $\pi\circ i_E=i_\Sg\circ(\pi|_E)$ implies that $\pi^*$ maps relative forms to relative forms. Pullback commutes with restriction, giving the commutative diagram
\[
\begin{tikzcd}[column sep=small,row sep=large]
0\arrow[r]&
\Omega(\fol,\Sg)\arrow[r]\arrow[d,"\pi^*"]&
\Omega(\fol|_U,\Sg)\oplus\Omega(\fol|_V)
  \arrow[r]\arrow[d,"\pi^*"]&
\Omega(\fol|_{U\cap V})\arrow[r]\arrow[d,"\pi^*"]&0\\
0\arrow[r]& \Omega(\blup{\Sg}{\fol},E)\arrow[r]& \Omega(\blup{\Sg}{\fol}|_{\pi^{-1}(U)},E)\oplus\Omega(\blup{\Sg}{\fol}|_{\pi^{-1}(V)})\arrow[r]&
\Omega(\blup{\Sg}{\fol}|_{\pi^{-1}(U\cap V)})\arrow[r]&0,
\end{tikzcd}
\]
where the top row the maps are $i_U^*\oplus i_V^*$ and $j_U^*-j_V^*$ (for the inclusions $j_U\colon U\cap V\to U$ and $j_V\colon U\cap V\to V$) and the bottom row uses the same restriction maps on the inverse images.

A standard argument proves exactness of each row: the map $i_U^*\oplus i_V^*$ is injective, and pairs in $\ker(j_U^*-j_V^*)$ glue to forms in $\Omega(\fol,\Sg)$. For $\eta\in\Omega(\fol|_{U\cap V})$, we have
\[j_U^*-j_V^*\Bigl(\operatorname{ext}_U((1-\varphi)\eta),-\operatorname{ext}_V(\varphi\eta)\Bigr)=(1-\varphi)\eta+\varphi\eta=\eta\]
(the first extension vanishes near $\Sg$ because $\operatorname{supp}(1-\varphi)\subset V$). The same argument with $\pi^*\varphi$ proves exactness of the bottom row.

The relative tube summands are acyclic, and $\pi$ is a foliated diffeomorphism over $V$. Therefore the middle and right vertical maps induce isomorphisms in cohomology. The five lemma applied to the associated long exact sequences gives the result.
\end{proof}

\begin{thm}\label{principal result}
The blow-down map induces a natural long exact sequence
\begin{equation}\label{blowup les}
\begin{tikzcd}[column sep=large]
  % Linha superior
  & \ \ \ \ \ \ \ \ \ \ \ \ \ \ \ \ \ \ \ \ \ \ \ \cdots \arrow[r] & \Ho^{i-1}(\blup{\Sg}{\fol}|_E) \arrow[dll, out=0, in=180, looseness=1.5, overlay, "\partial" description] \\
  % Linha do meio
  \Ho^i(\fol) \arrow[r, "{(\pi^*,\,i_\Sg^*)}"] & \Ho^i(\blup{\Sg}{\fol})\oplus\Ho^i(\fol|_\Sg) \arrow[r, "{i_E^*-(\pi|_E)^*}"] & \Ho^i(\blup{\Sg}{\fol}|_E) \arrow[dll, out=0, in=180, looseness=1.5, overlay, "\partial" description] \\
  % Linha inferior
  \Ho^{i+1}(\fol) \arrow[r]  & \cdots \ \ \ \ \ \ \ \ \ \ \ \ \ \ \ \ \ \ \ \ \ \ \  &
\end{tikzcd}
\end{equation}
If $\pi^*\colon\Ho^j(\fol)\to\Ho^j(\blup{\Sg}{\fol})$
is injective for $j=i,i+1$, then
\begin{equation}\label{equation isom cohom bas blow}
\Ho^i(\blup{\Sg}{\fol})\cong\pi^*(\Ho^i(\fol))\oplus
\frac{\Ho^i(\blup{\Sg}{\fol}|_E)}{(\pi|_E)^*(\Ho^i(\fol|_\Sg))}
\end{equation}
as vector spaces.
\end{thm}

\begin{proof}
Set
\[K=\bigl\{(\beta,\gamma)\in\Omega(\blup{\Sg}{\fol})\oplus\Omega(\fol|_\Sg) \mid i_E^*\beta=(\pi|_E)^*\gamma\bigr\}\]
and $\kappa=(\pi^*,i_\Sg^*)\colon\Omega(\fol)\to K$. Given $\gamma\in\Omega(\fol|_\Sg)$, surjectivity of $i_E^*$
gives a form $\beta$ with $i_E^*\beta=(\pi|_E)^*\gamma$. Thus $(\beta,\gamma)\in K$ and the second projection is surjective. The resulting diagram has short exact rows:
\[\begin{tikzcd}[row sep=large]
0\arrow[r]& \Omega(\fol,\Sg)\arrow[r,hook]\arrow[d,"\pi^*"]& \Omega(\fol)\arrow[r,"i_\Sg^*"]\arrow[d,"\kappa"]& \Omega(\fol|_\Sg)\arrow[r]\arrow[d,equal]&0\\
0\arrow[r]& \Omega(\blup{\Sg}{\fol},E)\arrow[r,hook]& K\arrow[r,"\operatorname{pr}_2"]& \Omega(\fol|_\Sg)\arrow[r]&0.
\end{tikzcd}\]
The lower inclusion sends $\beta$ to $(\beta,0)$. By Proposition~\ref{relative blowup iso} and the five lemma, $\Ho^i(\kappa)\colon\Ho^i(\fol)\to\Ho^i(K)$ is an isomorphism in every degree. Under this identification, \eqref{blowup les} is the long exact sequence of
\[0\longrightarrow K\longrightarrow \Omega(\blup{\Sg}{\fol})\oplus\Omega(\fol|_\Sg) \xrightarrow{i_E^*-(\pi|_E)^*}
\Omega(\blup{\Sg}{\fol}|_E)\longrightarrow0.\]
In particular, for a closed $\eta\in\Omega^i(\blup{\Sg}{\fol}|_E)$, choose $\beta\in\Omega^i(\blup{\Sg}{\fol})$ with $i_E^*\beta=\eta$. Then $(d\beta,0)\in K^{i+1}$ and
\[\partial[\eta]=\Ho^{i+1}(\kappa)^{-1}\bigl([(d\beta,0)]\bigr).\]

If $\pi^*$ is injective in degree $i+1$, then exactness in that degree gives
\[\im\partial=\ker(\pi^*,i_\Sg^*)\subseteq\ker\pi^*=\{0\}\]
and, consequently,
\[\Ho^i(\blup{\Sg}{\fol}|_E)=i_E^*\bigl(\Ho^i(\blup{\Sg}{\fol})\bigr)+(\pi|_E)^*\bigl(\Ho^i(\fol|_\Sg)\bigr).\]
Let $[i_E^*]$ denote restriction followed by the quotient by $(\pi|_E)^*(\Ho^i(\fol|_\Sg))$. The preceding equality gives its surjectivity, and exactness at the middle term of \eqref{blowup les} gives
\[\ker[i_E^*]=\operatorname{pr}_1\bigl(\ker(i_E^*-(\pi|_E)^*)\bigr)=\pi^*\bigl(\Ho^i(\fol)\bigr).\]
Injectivity of $\pi^*$ in degree $i$ now gives the short exact sequence
\begin{equation}\label{injective short exact}
0\longrightarrow\Ho^i(\fol)\xrightarrow{\pi^*}
\Ho^i(\blup{\Sg}{\fol})\xrightarrow{[i_E^*]}
\frac{\Ho^i(\blup{\Sg}{\fol}|_E)}{(\pi|_E)^*(\Ho^i(\fol|_\Sg))}
\longrightarrow0,
\end{equation}
hence \eqref{equation isom cohom bas blow}.
\end{proof}

\begin{cor}\label{basic euler blowup}
If $M$ is compact, the basic Euler characteristics satisfy
\[
\chi_{\mathrm{bas}}(\blup{\Sg}{\fol})-\chi_{\mathrm{bas}}(\fol)
=\chi_{\mathrm{bas}}(\blup{\Sg}{\fol}|_E)-\chi_{\mathrm{bas}}(\fol|_\Sg).
\]
\end{cor}

\begin{proof}
The four foliations are defined on compact manifolds, so their basic
cohomologies are finite dimensional \cite[Theorem~4.4]{equicaramello}.
The alternating sum of dimensions in \eqref{blowup les} gives
\[
0=\sum_i(-1)^i\Bigl(
\dim\Ho^i(\fol)-\dim\Ho^i(\blup{\Sg}{\fol})
-\dim\Ho^i(\fol|_\Sg)+\dim\Ho^i(\blup{\Sg}{\fol}|_E)
\Bigr),
\]
which is the stated identity.
\end{proof}

The proof of Theorem~\ref{principal result} works more generally for the blow-up along a closed saturated submanifold within a stratum, using its distinguished tubular neighborhood. This further generalization can be used, for example, when blowing-up an already regular Riemannian foliation along a closed saturated submanifold --- for instance, a stratum of $\overline{\fol}$ or a holonomy stratum. In particular, taking $\fol$ to be the foliation by points relates Theorem~\ref{principal result} to the classical blow-up formulas \cite{gitler,poag}.

\begin{cor}\label{de rham cohomology}
Let $M$ be a compact orientable manifold and let $\Sg\subset M$ be a closed submanifold of positive codimension such that $\blup{\Sg}{M}$ is orientable. Then
\[\Ho^i_{\mathrm{dR}}(\blup{\Sg}{M})\cong\Ho^i_{\mathrm{dR}}(M)\oplus\frac{\Ho^i_{\mathrm{dR}}(E)}{(\pi|_E)^*\Ho^i_{\mathrm{dR}}(\Sg))}.\]
\end{cor}

\begin{proof}
For compatible orientations, $\pi$ has degree one on each component since it is a diffeomorphism off $\Sg$. By Poincar\'e duality, the Gysin map is a left inverse of $\pi^*$ \cite{tueq}. Apply Theorem~\ref{principal result} to the foliation by points.
\end{proof}

There is a special situation in which the basic cohomology of the original foliation and that of its blow-up are naturally isomorphic.

\begin{prop}\label{pullback foliation}
Let $\pi\colon E\to B$ be a fiber bundle with connected fibers, and let $\fol$ be a regular foliation of $B$. On compatible total transversals, $\pi$ identifies the holonomy pseudogroups of $\pi^*(\fol)$ and $\fol$. In particular, $\pi^*\colon\Ho(\fol)\to\Ho(\pi^*(\fol))$ is an isomorphism of graded algebras.
\end{prop}

\begin{proof}
In this situation one can choose total transversals $S$ and $\hat{S}$ for $\fol$ and $\pi^*(\fol)$, respectively, such that the restriction
   \[\hat{\pi}:=\pi|_{\hat{S}}\colon \hat{S} \rar S\]
is a diffeomorphism (for instance, $\hat{S}$ can be constructed from a transversal $S$ whose components are contained in $E$-trivializing open sets, as intersections of $\pi^{-1}(S)$ with local sections on the trivializations). The pseudogroup equivalence $\mathscr{H}(\pi^*(\fol))\cong\mathscr{H}(\fol)$ is then induced by $\hat{\pi}$: it is the family of all restrictions $\pi|_W$, where $W$ is an open subset of $\hat{S}$. Indeed, $\mathscr{H}(\fol)$ is generated by elements of the form $\hat{\pi}|_{W'} \circ \hat{h} \circ (\hat{\pi}|_W)^{-1}$, for sufficiently small $W$ and $W'$ (see \cite[Chapter 1]{dynamics} for more details).

The space $\Omega^k(S)^{\mathscr{H}(\fol)}$ of holonomy-invariant $k$-forms on $S$ is naturally identified with the space of basic $k$-forms $\Omega^k(\fol)$ (see, e.g. \cite{haefliger}). Hence
\[\Omega^k(\pi^*(\fol))\cong\Omega^k(\hat{S})^{\mathscr{H}(\pi^*(\fol))}\cong\Omega^k(S)^{\mathscr{H}(\fol)}\cong\Omega^k(\fol).\]
where the middle isomorphism is induced by $\pi$, as we saw, so $\pi^*$ is an isomorphism already on the level of basic forms.
\end{proof}

\begin{thm}\label{cor pullback foliation}
If $\blup{\Sg}{\fol}|_E=(\pi|_E)^*(\fol|_\Sg)$, then
\[\pi^*\colon\Ho(\fol)\longrightarrow\Ho(\blup{\Sg}{\fol})\]
is an isomorphism of graded algebras.
\end{thm}

\begin{proof}
By Proposition~\ref{pullback foliation}, $(\pi|_E)^*$ is an isomorphism in every degree. Hence $i_E^*-(\pi|_E)^*$ in
\eqref{blowup les} is surjective in every degree, and the connecting maps vanish. By exactness, $(\pi^*,i_\Sg^*)$ identifies $\Ho^i(\fol)$ with
\[\ker\bigl(i_E^*-(\pi|_E)^*\bigr)=\left\{\bigl(b,((\pi|_E)^*)^{-1}i_E^*b\bigr)  \mid b\in\Ho^i(\blup{\Sg}{\fol})\right\}.\]
The first projection of this graph is an isomorphism, and its composition with $(\pi^*,i_\Sg^*)$ is $\pi^*$. Since pullback preserves the wedge product, $\pi^*$ is an isomorphism of graded algebras.
\end{proof}

For $x\in\Sg$, write $E_x=(\pi|_E)^{-1}(x)$, and let $\hat L_{\hat x}$ be the leaf of $\blup{\Sg}{\fol}|_E$ through
$\hat x\in E_x$. Since $\pi|_E$ is foliated,
\begin{equation}\label{eq: continencia distribuição blowup} \dif(\pi|_E)_{\hat x}(T_{\hat x}\hat L_{\hat x})\subseteq T_xL_x.\end{equation}
For a field $X$ tangent to $\fol$, its lift is tangent to $\blup{\Sg}{\fol}$ by Proposition~\ref{propositon: blow up of foliate}
and satisfies
\[\dif(\pi|_E)_{\hat x}\bigl((\blup{\Sg}{X})_{\hat x}\bigr)=X_x.\]
As the values $X_x$ span $T_xL_x$, equality holds in \eqref{eq: continencia distribuição blowup}. Since $\ker\dif(\pi|_E)_{\hat x}=T_{\hat x}E_x$, we obtain
\[\bigl(\dif(\pi|_E)_{\hat x}\bigr)^{-1}(T_xL_x)=T_{\hat x}\hat L_{\hat x}+T_{\hat x}E_x.\]
Thus $\blup{\Sg}{\fol}|_E=(\pi|_E)^*(\fol|_\Sg)$ whenever
$T_{\hat x}E_x\subseteq T_{\hat x}\hat L_{\hat x}$ for every $\hat x\in E$.

\begin{cor}\label{blowup low codimension}
If every component of the minimal locus $\Sg$ has codimension two or three, then $\blup{\Sg}{\fol}|_E=(\pi|_E)^*(\fol|_\Sg)$. The pullback $\pi^*$ is an isomorphism of graded basic cohomology algebras, and $\pi$ induces homeomorphisms of the spaces of leaves and of leaf closures.
\end{cor}

\begin{proof}
Fix $x\in\Sg$ and let $\fol_x$ be the infinitesimal foliation on $\nu_xL_x$ obtained from the normal slice by homotheties
\cite[Section~2]{radeschislice}. In the orthogonal splitting
\[\nu_xL_x=(T_x\Sg\cap\nu_xL_x)\oplus\nu_x\Sg,\]
$\fol_x$ is the product of the point foliation on the first factor and $\fol_x|_{\nu_x\Sg}$ on the second
\cite[Remark~2.5]{molinoconjec}. The leaves of $\fol_x$ lie in spheres centered at the origin, so the unit sphere $S(\nu_x\Sg)$ is saturated and carries an induced singular Riemannian foliation.

Let $C$ be a leaf of this sphere foliation, let $v\in C$, and set $y=\exp_x(rv)$ for sufficiently small $r>0$. Write
$S_x=\exp_x(B_\varepsilon(\nu_xL_x))$ for a normal slice. Tangent fields whose values at $x$ form a basis of $T_xL_x$
still span a complement to the slice near $x$. Thus $T_yL_y+T_yS_x=T_yM$. By the definition of $\fol_x$, $v\mapsto\exp_x(rv)$ identifies the germ of $C$ at $v$ with that of the slice plaque at $y$. Hence
\[\dim C=\dim(T_yL_y\cap T_yS_x)=\dim L_y-\dim L_x>0,\]
since $y\in M\setminus\Sg$ and $L_x$ has minimal dimension.

In codimension two, every leaf of $S(\nu_x\Sg)\cong S^1$ has dimension one, so each leaf is open and the sphere is a single leaf. In codimension three, suppose $S(\nu_x\Sg)\cong S^2$ has a one-dimensional leaf. The normal exponential map of a small plaque in that leaf gives an open set swept out by geodesics normal to the plaque. At each point of this open set, the leaf tangent space is perpendicular to a nonzero geodesic velocity by transnormality, so it has dimension at most one. Density of the regular locus then forces every leaf to have dimension one. The resulting tangent line bundle is trivial because $S^2$ is simply connected. A nowhere-zero section would give $\chi(S^2)=0$ by Poincar\'e--Hopf, in contradiction. Thus every leaf is two-dimensional and open, so $S^2$ is a single leaf.

For small $r>0$, $\exp_x(rC)$ lies in one leaf of $\fol$. Since $\pi$ is a foliated diffeomorphism off $E$, the set
$\{[x,v,r]\mid v\in C\}$ lies in one leaf of $\blup{\Sg}{\fol}$. The foliated retraction $[x,v,r]\mapsto[x,v,0]$ sends this set into one leaf of $\blup{\Sg}{\fol}|_E$. Since $C=S(\nu_x\Sg)$, the map $v\mapsto[x,v,0]$ is the double covering of $E_x$ by the unit normal sphere. Its differential gives
\[T_{\hat x}E_x\subseteq T_{\hat x}\hat L_{\hat x}\]
for every $\hat x\in E_x$ This proves the asserted equality on $E$, and Theorem~\ref{cor pullback foliation} gives the cohomology isomorphism.

Let $L$ be a leaf of $\fol$ and $\hat L$ the leaf through a lift of a point of $L$, so $\pi(\hat L)=L$. Since $\pi$ is continuous and closed, also $\pi(\overline{\hat L})=\overline L$. Both $\hat L$ and $\overline{\hat L}$ are saturated by leaves and hence by fibers of $\pi$, therefore
\[\pi^{-1}(L)=\hat L\quad \text{and}\quad \pi^{-1}(\overline L)=\overline{\hat L}.\]
These give bijections on both quotient spaces. The composition of the proper surjection $\pi$ with either quotient
projection downstairs is a quotient map, so both bijections are homeomorphisms.
\end{proof}

\begin{cor}\label{blowup low basic degrees}
For a blow-up along the minimal locus, the pullback
\[\pi^*\colon\Ho^i(\fol)\longrightarrow\Ho^i(\blup{\Sg}{\fol})\]
is an isomorphism for $i=1$ and is injective for $i=2$. The same
conclusions hold for the pullback $\Pi^*$ of the full desingularization.
\end{cor}

\begin{proof}
Let $\alpha\in\Omega^1(\blup{\Sg}{\fol}|_E)$ and $\beta\in\Omega^2(\fol|_\Sg)$ satisfy $d\alpha=(\pi|_E)^*\beta$. Since $\pi|_E$ is constant on $E_x$,
\[d(\alpha|_{E_x})=\bigl((\pi|_E)^*\beta\bigr)|_{E_x}=0.\]
If the component through $x$ has codimension at least three, then $E_x$ is a real projective space of dimension at least two and $\Ho^1_{\mathrm{dR}}(E_x)=0$. In codimension two, the preceding fiberwise argument gives $T_{\hat x}E_x\subseteq
T_{\hat x}\hat L_{\hat x}$ for every $\hat x\in E_x$, so $\alpha|_{E_x}=0$. Thus $\alpha|_{E_x}$ is exact for every $x$.

We now equip $E_x=\mathbb P(\nu_x\Sg)$ with the metric induced by the normal metric, and denote its volume density by $d\mu_x$. Let us define $f$ on $E$ by $d(f|_{E_x})=\alpha|_{E_x}$ and
\[\int_{E_x}f\,d\mu_x=0,\]
which determine $f|_{E_x}$ because $E_x$ is connected. In a local trivialization, integration from a fixed point of the model fiber gives a primitive of $\alpha|_{E_x}$. Near each endpoint, the path may be chosen as a fixed path followed by a segment in a coordinate chart, so one sees that the integral depends smoothly on the base and the endpoint. Subtracting its fiberwise mean gives the smooth function $f$.

Let $\phi_s$ be the local flow of a field $X$ tangent to $\fol$ and let $\hat\phi_s$ be its lift. We write $A_s(x)\colon\nu_x\Sg\to\nu_{\phi_s(x)}\Sg$ for the normal derivative of $\phi_s$. In a local orthonormal trivialization of the normal bundle, the normal component of $\phi_s(\exp_x(tw))$ is $tA_s(x)w+O(t^2)$. Since the flow preserves the distance cylinders about $\Sg$ \cite[Proposition~2.7]{desmarcos},
\[t^2\|A_s(x)w\|^2+O(t^3)=d_\metric\bigl(\phi_s(\exp_x(tw)),\Sg\bigr)^2=t^2\|w\|^2.\]
Comparing the coefficients of $t^2$ shows that $A_s(x)$ is an isometry. Hence $\hat\phi_s|_{E_x}=\mathbb P(A_s(x))$ is an
isometry onto $E_{\phi_s(x)}$. The generator of $\hat\phi_s$ is tangent to $\blup{\Sg}{\fol}|_E$ on $E$, so
$\hat\phi_s^*\alpha=\alpha$ there. It follows that $d\bigl((f\circ\hat\phi_s)|_{E_x}\bigr)=\alpha|_{E_x}$ and
\[\int_{E_x}(f\circ\hat\phi_s)\,d\mu_x=\int_{E_{\phi_s(x)}}f\,d\mu_{\phi_s(x)}=0.\]
Uniqueness of the normalized primitive gives $f\circ\hat\phi_s=f$, and differentiating at $s=0$ yields
$(\blup{\Sg}{X}|_E)(f)=0$.

We now set $\theta=\alpha-df$. Then, for every vertical field $Y$ on $E$,
\[\iota_Y\theta=0,\quad \iota_Yd\theta=\iota_Y(\pi|_E)^*\beta=0\quad\text{and}\quad \mathcal L_Y\theta=0.\]
In bundle charts, these identities say that $\theta$ has only base components and their coefficients are constant on the connected fibers. Hence $\theta=(\pi|_E)^*\eta$ for a unique $\eta\in\Omega^1(\Sg)$. Pullback by a surjective submersion is injective on forms, so
\[(\pi|_E)^*(d\eta-\beta)=d\theta-(\pi|_E)^*\beta=0,\]
yielding $d\eta=\beta$. If $X$ is tangent to $\fol$, then
\[(\pi|_E)^*\bigl(\eta(X|_\Sg)\bigr)=\theta\bigl(\blup{\Sg}{X}|_E\bigr)=0.\]
The restrictions $X|_\Sg$ span $T(\fol|_\Sg)$. Thus $\eta$ and $d\eta=\beta$ annihilate the leaf directions, so $\eta$ is basic. Then $df=\alpha-(\pi|_E)^*\eta$ is basic, which implies that $f$ is basic.

If $d\alpha=0$, we obtain $[\alpha]=(\pi|_E)^*[\eta]$, giving surjectivity in degree one. If $(\pi|_E)^*[\beta]=0$ in degree two, choose a basic $\alpha$ with $d\alpha=(\pi|_E)^*\beta$. Then $\beta=d\eta$ with $\eta$ basic, giving injectivity in degree two. For injectivity in degree one, suppose that $(\pi|_E)^*\eta=dg$ with $g$ basic. Since $dg$ annihilates the vertical directions and the fibers are connected, $g=(\pi|_E)^*h$ for a smooth function $h$ on $\Sg$. Then $dh=\eta$, and $h$ is basic. The same descent applied to $dg=0$ gives surjectivity in degree zero, where injectivity follows from surjectivity of $\pi|_E$.

For $i=1,2$, surjectivity of $(\pi|_E)^*$ in degree $i-1$ makes $i_E^*-(\pi|_E)^*$ surjective in that degree. Exactness in
\eqref{blowup les} gives
\[(\pi^*,i_\Sg^*)\colon\Ho^i(\fol)\xrightarrow{\;\cong\;}\ker\bigl(i_E^*-(\pi|_E)^*\bigr),\]
where the kernel is taken in degree $i$. In degree one, this kernel is the graph of $((\pi|_E)^*)^{-1}\circ i_E^*$, so its
first projection gives the isomorphism $\pi^*$. In degree two, if $\pi^*a=0$, then
\[(\pi|_E)^*i_\Sg^*a=i_E^*\pi^*a=0.\]
Injectivity of $(\pi|_E)^*$ gives $i_\Sg^*a=0$, hence $a=0$ by injectivity of $(\pi^*,i_\Sg^*)$. Composing the pullbacks at the successive blow-ups gives the assertions for $\Pi^*$.
\end{proof}

\begin{example}\label{blowup degree three example}
Consider the orbit foliation of the scalar $S^1$-action on $S^4\subset\mathbb C^2\times\RR$ given by
\[\lambda\cdot(z,t)=(\lambda z,t).\]
The minimal locus consists of the two fixed poles. The regular level spheres are foliated by the Hopf circles, so the original leaf space is the suspension of $\mathbb{CP}^1$. After blowing up the poles, each exceptional $\mathbb{RP}^3$ carries the projective Hopf foliation. Since the antipodal map preserves each Hopf circle, the blown-up leaf
space is $\mathbb{CP}^1\times[-1,1]$.

For the basic cohomology calculation, cover $S^4$ by the complements of the two poles. Each open set retracts by a foliated homotopy onto its remaining pole, and the overlap has the basic cohomology of $\mathbb{CP}^1$. Basic Mayer--Vietoris gives $\Ho^3(\fol)\cong\Ho^2(\mathbb{CP}^1;\RR)\cong\RR$. Upstairs, the two inverse images retract onto their exceptional divisors, both with the basic cohomology of $\mathbb{CP}^1$. The restriction to the overlap from either tube induces an isomorphism in degree two, since the maps to $\mathbb{CP}^1$ agree with $(z,t)\mapsto[z]_{\mathbb C}$ there. Thus $\Ho^3(\blup{\Sg}{\fol})=0$. In particular, the pullback can fail to be injective in degree three.
\end{example}

\section*{Acknowledgements}
We are grateful to Prof. Dirk Töben and Prof. Marcos Alexandrino for the insightful suggestions and helpful discussions. LRS was financed in part by the Coordenação de Aperfeiçoamento de Pessoal de Nível Superior - Brazil (CAPES) – Finance Code 001 and in part by Fundação de Amparo à Pesquisa do Estado de São Paulo (FAPESP), grant 2025/27715-7.

\end{document}